\documentclass[10pt]{amsart}
\usepackage[margin=1.2in]{geometry}

 \usepackage[UKenglish]{babel}
\usepackage{graphicx}	
\usepackage{amsthm}
\usepackage{amsmath}
\usepackage{amssymb}
\usepackage{mathtools}
\usepackage{enumerate}
\usepackage{mathrsfs}
\usepackage{thmtools, thm-restate}
\usepackage[table]{xcolor}
\usepackage{tabularx, booktabs}
\newcolumntype{Y}{>{\centering\arraybackslash}X}
\usepackage{tikz}
\usepackage{dsfont}
\usepackage{pgfplots}
\usepackage{wrapfig}
\usepackage{caption}
\usepackage{hyperref}
\usepackage{longtable}

\usepackage{tikz-cd} 
\usetikzlibrary {arrows.meta,bending,positioning,patterns,decorations}

\usepackage[OT2,T1]{fontenc}
\DeclareSymbolFont{cyrletters}{OT2}{wncyr}{m}{n}
\DeclareMathSymbol{\Sha}{\mathalpha}{cyrletters}{"58}

\newtheorem{theorem}{Theorem}[section]
\newtheorem{corollary}[theorem]{Corollary}
\newtheorem{conjecture}[theorem]{Conjecture}
\newtheorem{proposition}[theorem]{Proposition}
\newtheorem{lemma}[theorem]{Lemma}

\theoremstyle{definition}
\newtheorem{definition}[theorem]{Definition}
\newtheorem{notation}[theorem]{Notation}
\newtheorem{example}[theorem]{Example}
\newtheorem{remark}[theorem]{Remark}

\newcommand{\Q}{\mathbb{Q}}
\newcommand{\R}{\mathbb{R}}
\newcommand{\C}{\mathbb{C}}
\newcommand{\Z}{\mathbb{Z}}
\newcommand{\cR}{\mathcal{R}}
\newcommand{\F}{\mathbb{F}}
\newcommand{\K}{\mathcal{K}}

\newcommand{\Frob}{\mathrm{Frob}}

\newcommand{\rank}{\mathrm{rk}}
\newcommand{\rk}{\mathrm{rk}}
\newcommand{\Gal}{\mathrm{Gal}}

\newcommand{\Jac}{\mathrm{Jac}\,}
\newcommand{\s}{\mathfrak{s}}
\newcommand{\ord}{\mathrm{ord}}

\newcommand{\coker}{\mathrm{coker}}
\newcommand{\E}{H_{f,\mathcal K}}
\newcommand{\Er}{H_{f,\R}}
\newcommand{\Ec}{H_{f,\C}}

\newcolumntype{C}[1]{>{\centering\let\newline\\\arraybackslash\hspace{0pt}}m{#1}}

\usepackage{pbox}
\usepackage{tkz-graph}
\usetikzlibrary{snakes,fit,positioning,calc,shapes}
\usepackage{graphicx}
\usepackage{cleveref,tikz-cd,pbox,wrapfig}

\definecolor{amethyst}{rgb}{0.6, 0.4, 0.8}
\definecolor{atomictangerine}{rgb}{1.0, 0.6, 0.4}
\definecolor{deeppeach}{rgb}{1.0, 0.8, 0.64}
\definecolor{eggshell}{rgb}{0.94, 0.92, 0.84}
\definecolor{lightapricot}{rgb}{0.99, 0.84, 0.69}
\definecolor{lemonchiffon}{rgb}{1.0, 0.98, 0.8}
\definecolor{roundabout}{rgb}{1.0, 0.91, 0.75}
\definecolor{atomictangerine}{rgb}{1.0, 0.6, 0.4}
\definecolor{ruby}{rgb}{0.88, 0.07, 0.37}
\definecolor{sapphire}{rgb}{0.03, 0.15, 0.4}

\def\rootsep{0.03}               
\def\clustersep{0.06}            
\def\cnamescale{0.4}             
\def\cdepthscale{0.4}            
\def\cltopskip{1pt}              
\def\clbottomskip{1pt}           

\def\rootscale{0.5}   \def\rootcolor{gray}
\def\rootscaleA{0.7}  \def\rootcolorA{yellow}
\def\rootscaleB{0.5}  \def\rootcolorB{green}
\def\rootscaleC{0.4}  \def\rootcolorC{sapphire}
\def\rootscaleD{0.45}  \def\rootcolorD{ruby}
\tikzset{
  clA/.style = {very thick,black},
  clB/.style = {thick,purple}
}

\def\graphdslabelscale{0.6}

\def\GraphScale{0.6}

\tikzset{
  root/.style = {circle,scale=\rootscale,fill=\rootcolor},
    rc/.style 2 args = {right=#1*1.5*\clustersep of {#2.east|-first},root}, rr/.style = {right=\rootsep of {#1.east|-first},root},
  roott/.style = {circle,inner sep=-2pt,minimum size=5pt,black,font=\ttfamily\footnotesize},
    rct/.style 2 args = {right=#1*1.5*\clustersep of {#2.east|-first},roott}, rrt/.style = {right=\rootsep of {#1.east|-first},roott},
  rootA/.style = {circle,scale=\rootscaleA,ball color=\rootcolorA},
    rcA/.style 2 args = {right=#1*1.5*\clustersep of {#2.east|-first},rootA}, rrA/.style = {right=\rootsep of {#1.east|-first},rootA},
  rootB/.style = {circle,scale=\rootscaleB,ball color=\rootcolorB},
    rcB/.style 2 args = {right=#1*1.5*\clustersep of {#2.east|-first},rootB}, rrB/.style = {right=\rootsep of {#1.east|-first},rootB},
  rootC/.style = {diamond,scale=\rootscaleC,ball color=\rootcolorC},
    rcC/.style 2 args = {right=#1*1.5*\clustersep of {#2.east|-first},rootC}, rrC/.style = {right=\rootsep of {#1.east|-first},rootC},
  rootD/.style = {circle,scale=\rootscaleD,ball color=\rootcolorD},
    rcD/.style 2 args = {right=#1*1.5*\clustersep of {#2.east|-first},rootD}, rrD/.style = {right=\rootsep of {#1.east|-first},rootD},
  cluster/.style = {draw=black!90,thick,rounded corners,inner sep=22*\clustersep,outer xsep=22*\clustersep,fit=#1},
  clabel/.style  = {anchor=west,scale=\cdepthscale,black,inner sep=0,outer xsep=1,outer ysep=0},
  clabelL/.style = {above right=-\clustersep of #1t.north east,clabel},
  clabelD/.style = {below right=-\clustersep of #1t.south east,clabel},
  clouter/.style = {inner sep=0,outer sep=0,fit=#1}
}

\def\Cluster #1 = #2;{\node[cluster=#2] (#1) {};}
\def\ClusterL #1[#2] = #3;{
  \node[cluster=#3] (#1t) {}; \node[clabelL=#1] (#1l) {$#2$}; \node[clouter=(#1t)(#1l)] (#1) {};}
\def\ClusterD #1[#2] = #3;{
  \node[cluster=#3] (#1t) {}; \node[clabelD=#1] (#1d) {$#2$}; \node[clouter=(#1t)(#1d)] (#1) {};}
\def\ClusterLD #1[#2][#3] = #4;{
  \node[cluster=#4] (#1t) {}; \node[clabelL=#1] (#1l) {$#2$}; 
  \node[clabelD=#1] (#1d) {$#3$}; \node[clouter=(#1t)(#1l)(#1d)] (#1) {};}
\def\ClusterLDName #1[#2][#3][#4] = #5;{
  \node[cluster=#5] (#1t) {}; \node[clabelL=#1] (#1l) {$#2$}; 
  \node[clabelD=#1] (#1d) {$#3$}; 
  \node[scale=\cnamescale,above=\clustersep/3 of #1t,inner sep=0, outer sep=0] (#1n) {$#4$}; 
  \node[clouter=(#1l)(#1d)(#1t)] (#1) {};}

\newcommand{\Root}[4][]{
  \ifx\relax#2\relax\node[rr#1=#3] (#4) {};\else\node[rc#1={#2}{#3}] (#4) {};\fi}
\newcommand{\RootT}[5][]{
  \ifx\relax#2\relax\node[rrt#1=#3] (#4) {#5};\else\node[rct#1={#2}{#3}] (#4) {#5};\fi}

\def\frob(#1)(#2){\path[draw,thick,shorten <=-22*\clustersep,shorten >=-22*\clustersep](#1.east)--(#2.west|-#1){};}

\usepackage{pbox}
\def\pb#1{\pbox[c]{\textwidth}{\hfil #1\hfil}}

\long\def\clusterpicture#1\endclusterpicture{\pb{\vbox to \cltopskip{\vfill}\\%
  \begin{tikzpicture}\node[coordinate] (first) {};#1\end{tikzpicture}\\[-11pt]\vbox to \clbottomskip{\vfill}}}   
\long\def\clusterpictureopt#1#2\endclusterpicture{\pb{\vbox to \cltopskip{\vfill}\\%
  \begin{tikzpicture}[#1]\node[coordinate] (first) {};#2\end{tikzpicture}\\[-11pt]\vbox to \clbottomskip{\vfill}}}

\def\pb#1{\pbox[c]{\textwidth}{\hfil #1\hfil}}

\def\GraphVertices{\SetVertexNormal[Shape=circle, FillColor=blue!50, LineColor=blue!50, LineWidth=0.8pt]
  \tikzset{VertexStyle/.append style = {inner sep=0.5pt,minimum size=0.3em,font = \tiny\bfseries}}}

\def\BlueEdges{  \SetUpEdge[lw=0.8pt,color=blue!70]
   \tikzset{EdgeStyle/.append style = {shorten <=0.5pt,shorten >=0.5pt}}}

\def\LoopW(#1){
  \path[draw,-,thick,color=blue!70] (#1) edge[out=155,in=90] ($(#1)-(1.3,0)$);
  \path[draw,-,thick,color=blue!70] (#1) edge[out=210,in=270] ($(#1)-(1.3,0)$);
}
\def\LoopE(#1){
  \path[draw,-,thick,color=blue!70] (#1) edge[out=25,in=90] ($(#1)+(1.3,0)$);
  \path[draw,-,thick,color=blue!70] (#1) edge[out=-25,in=270] ($(#1)+(1.3,0)$);
}
\def\LoopS(#1){
  \path[draw,-,thick,color=blue!70] (#1) edge[out=115,in=180] ($(#1)+(0,1.2)$);
  \path[draw,-,thick,color=blue!70] (#1) edge[out=65,in=0] ($(#1)+(0,1.2)$);
}
\def\LoopN(#1){
  \path[draw,-,thick,color=blue!70] (#1) edge[out=-115,in=180] ($(#1)-(0,1.2)$);
  \path[draw,-,thick,color=blue!70] (#1) edge[out=-65,in=0] ($(#1)-(0,1.2)$);
}
\def\EdgeW(#1){
  \path[draw,-,thick,color=blue!70] (#1+) edge[out=180,in=90] ($(#1+)-(1.3,0.3)$);
  \path[draw,-,thick,color=blue!70] (#1-) edge[out=180,in=270] ($(#1+)-(1.3,0.3)$);
}
\def\EdgeE(#1){
  \path[draw,-,thick,color=blue!70] (#1+) edge[out=0,in=90] ($(#1-)+(1.3,0.3)$);
  \path[draw,-,thick,color=blue!70] (#1-) edge[out=0,in=270] ($(#1-)+(1.3,0.3)$);
}
\def\EdgeS(#1){
  \path[draw,-,thick,color=blue!70] (#1+) edge[out=90,in=0] ($(#1-)+(0.3,1.3)$);
  \path[draw,-,thick,color=blue!70] (#1-) edge[out=90,in=180] ($(#1-)+(0.3,1.3)$);
}
\def\EdgeN(#1){
  \path[draw,-,thick,color=blue!70] (#1+) edge[out=270,in=0] ($(#1+)-(0.3,1.3)$);
  \path[draw,-,thick,color=blue!70] (#1-) edge[out=270,in=180] ($(#1+)-(0.3,1.3)$);
}
\def\GCircle(#1,#2)(#3,#4){
  \path(#1,#2) node[coordinate] (1) {};
  \path(#3,#4) node[coordinate] (2) {};
  \path[draw,-,thick,color=blue!70] (1) edge[out=90,in=90] (2);
  \path[draw,-,thick,color=blue!70] (2) edge[out=270,in=270] (1);
}

\def\EdgeSign(#1)(#2)#3(#4)#5{
  \node at ($(#1)!#3!(#2) + (#4)$) [color=black, scale=\graphdslabelscale] {$\scriptstyle #5$};
}

\def\GraphEdgeSignDist{0.55}

\def\GraphEdgeSignS(#1)(#2)#3#4{\EdgeSign(#1)(#2)#3(0,-\GraphEdgeSignDist){#4}}

\def\VSwap#1#2#3#4{\path[draw](#1) edge[<->,#3,shorten >=#4pt,shorten <=#4pt] (#2){};}
\def\VArr#1#2#3#4{\path[draw](#1) edge[->,#3,shorten >=#4pt,shorten <=#4pt] (#2){};}

\def\ESwapOfs#1#2#3#4#5#6#7#8{\VSwap{$(#1)!0.5!(#2) + (#6)$}{$(#3)!0.5!(#4) + (#7)$}{#5}{#8}}

\def\EArrOfs#1#2#3#4#5#6#7#8{\VArr{$(#1)!0.5!(#2) + (#6)$}{$(#3)!0.5!(#4) + (#7)$}{#5}{#8}}

\def\goodE{\clusterpicture            
  \Root[D] {1} {first} {r1};
  \Root[D] {} {r1} {r2};
  \Root[D] {} {r2} {r3};
  \ClusterD c1[0] = (r1)(r2)(r3);
\endclusterpicture}

\def\tpE{\clusterpicture            
  \Root[D] {1} {first} {r2};
  \Root[D] {} {r2} {r3};
  \ClusterLD c1[{+}][{\frac{n}{2}}] = (r2)(r3);
  \Root[D] {1} {c1} {r4};
  \ClusterD c3[0] = (r2)(c1)(r4);
\endclusterpicture}

\def\tmE{\clusterpicture            
  \Root[D] {1} {first} {r2};
  \Root[D] {} {r2} {r3};
  \ClusterLD c1[{-}][{\frac{n}{2}}] = (r2)(r3);
  \Root[D] {1} {c1} {r4};
  \ClusterD c3[0] = (r2)(c1)(r4);
\endclusterpicture}

\def\goodEprime{\clusterpicture            
  \Root[C] {1} {first} {r1};
  \Root[D] {} {r1} {r2};
  \Root[D] {} {r2} {r3};
    \Root[D] {} {r3} {r4};
  \ClusterD c1[0] = (r1)(r2)(r3)(r4);
\endclusterpicture}

\def\ctpEprime{\clusterpicture            
  \Root[C] {1} {first} {r2};
  \Root[D] {} {r2} {r3};
  \ClusterLD c1[{+}][n] = (r2)(r3);
  \Root[D] {1} {c1} {r4};
  \Root[D] {} {r4} {r5};
  \ClusterD c3[0] = (r2)(c1)(r4)(r5);
\endclusterpicture}

\def\ctmEprime{\clusterpicture            
  \Root[C] {1} {first} {r2};
  \Root[D] {} {r2} {r3};
  \ClusterLD c1[{-}][n] = (r2)(r3);
  \Root[D] {1} {c1} {r4};
  \Root[D] {} {r4} {r5};
  \ClusterD c3[0] = (r2)(c1)(r4)(r5);
\endclusterpicture}

\def\ctEprime{\clusterpicture            
  \Root[C] {1} {first} {r2};
  \Root[D] {} {r2} {r3};
  \ClusterLD c1[+][1] = (r2)(r3);
  \Root[D] {1} {c1} {r4};
  \Root[D] {} {r4} {r5};
  \ClusterD c3[0] = (r2)(c1)(r4)(r5);
\endclusterpicture}

\def\btpEprime{\clusterpicture            
  \Root[C] {1} {first} {r5};
  \Root[D] {} {r5} {r6};
  \Root[D] {2} {r6} {r1};
  \Root[D] {} {r1} {r2};
  \ClusterLD c1[{+}][\frac{n}{2}] = (r1)(r2);
  \ClusterD c2[0] = (r5)(r6)(c1);
\endclusterpicture}

\def\btmEprime{\clusterpicture            
  \Root[C] {1} {first} {r5};
  \Root[D] {} {r5} {r6};
  \Root[D] {2} {r6} {r1};
  \Root[D] {} {r1} {r2};
  \ClusterLD c1[{-}][\frac{n}{2}] = (r1)(r2);
  \ClusterD c2[0] = (r5)(r6)(c1);
\endclusterpicture}

\def\Jgood{\clusterpicture            
  \Root {1} {first} {r1};
  \Root {} {r1} {r2};
  \Root {} {r2} {r3};
  \ClusterD c1[0] = (r1)(r2)(r3);
\endclusterpicture}

\def\JCtwinP{\clusterpicture            
  \Root {1} {first} {r2};
  \Root {} {r2} {r3};
  \ClusterLD c1[+][n] = (r2)(r3);
  \Root {1} {c1} {r4};
  \ClusterD c3[0] = (r2)(c1)(r4);
\endclusterpicture}

\def\JMtwinP{\clusterpicture            
  \Root {1} {first} {r2};
  \Root {} {r2} {r3};
  \ClusterLD c1[+][\frac{n}{2}] = (r2)(r3);
  \Root {1} {c1} {r4};
  \ClusterD c3[0] = (r2)(c1)(r4);
\endclusterpicture}

\def\JCtwinM{\clusterpicture            
  \Root {1} {first} {r2};
  \Root {} {r2} {r3};
  \ClusterLD c1[-][n] = (r2)(r3);
  \Root {1} {c1} {r4};
  \ClusterD c3[0] = (r2)(c1)(r4);
\endclusterpicture}

\def\JMtwinM{\clusterpicture            
  \Root {1} {first} {r2};
  \Root {} {r2} {r3};
  \ClusterLD c1[-][\frac{n}{2}] = (r2)(r3);
  \Root {1} {c1} {r4};
  \ClusterD c3[0] = (r2)(c1)(r4);
\endclusterpicture}

\def\JCtwin{\clusterpicture            
  \Root {1} {first} {r2};
  \Root {} {r2} {r3};
  \ClusterLD c1[+][1] = (r2)(r3);
  \Root {1} {c1} {r4};
  \ClusterD c3[0] = (r2)(c1)(r4);
\endclusterpicture}

\def\tgrGB{\raise-7pt\hbox{\begin{tikzpicture}[scale=\GraphScale]
  \GraphVertices
  \Vertex[x=1.50,y=0.000,L=1]{1};
  \coordinate (2) at (0.000,0.000);
  \BlueEdges
  \LoopW(1)
\GraphEdgeSignS(1)(2){0.5}{n}\end{tikzpicture}}}

\def\tgrGBex{\raise-7pt\hbox{\begin{tikzpicture}[scale=\GraphScale]
  \GraphVertices
  \Vertex[x=1.50,y=0.000,L=1]{1};
  \coordinate (2) at (0.000,0.000);
  \BlueEdges
  \LoopW(1)
\GraphEdgeSignS(1)(2){0.5}{1}\end{tikzpicture}}}

\def\tgrGC{\raise-7pt\hbox{\begin{tikzpicture}[scale=\GraphScale]
  \GraphVertices
  \Vertex[x=1.50,y=0.000,L=1]{1};
  \coordinate (2) at (0.000,0.000);
  \BlueEdges
  \LoopW(1)
\GraphEdgeSignS(1)(2){0.5}{n}\ESwapOfs1212{}{0,-0.25}{0,0.25}{0.5}\end{tikzpicture}}}

\def\tgrGD{\raise-7pt\hbox{\begin{tikzpicture}[scale=\GraphScale]
  \GraphVertices
  \Vertex[x=1.50,y=0.000,L=\relax]{1};
  \coordinate (2) at (3.00,0.000);
  \coordinate (3) at (0.000,0.000);
  \BlueEdges
  \LoopE(1)
  \LoopW(1)
\GraphEdgeSignS(1)(3){0.5}{n}\GraphEdgeSignS(1)(2){0.5}{n}\end{tikzpicture}}}

\def\tgrGE{\raise-7pt\hbox{\begin{tikzpicture}[scale=\GraphScale]
  \GraphVertices
  \Vertex[x=1.50,y=0.000,L=\relax]{1};
  \coordinate (2) at (3.00,0.000);
  \coordinate (3) at (0.000,0.000);
  \BlueEdges
  \LoopE(1)
  \LoopW(1)
\GraphEdgeSignS(1)(3){0.5}{n}\GraphEdgeSignS(1)(2){0.5}{n}\ESwapOfs1212{}{0,-0.25}{0,0.25}{0.5}\end{tikzpicture}}}

\def\tgrGF{\raise-7pt\hbox{\begin{tikzpicture}[scale=\GraphScale]
  \GraphVertices
  \Vertex[x=1.50,y=0.000,L=\relax]{1};
  \coordinate (2) at (3.00,0.000);
  \coordinate (3) at (0.000,0.000);
  \BlueEdges
  \LoopE(1)
  \LoopW(1)
\GraphEdgeSignS(1)(3){0.5}{n}\GraphEdgeSignS(1)(2){0.5}{n}\ESwapOfs1313{}{0,-0.25}{0,0.25}{0.5}\ESwapOfs1212{}{0,-0.25}{0,0.25}{0.5}\end{tikzpicture}}}

\def\tgrGG{\raise-7pt\hbox{\begin{tikzpicture}[scale=\GraphScale]
  \GraphVertices
  \Vertex[x=1.50,y=0.000,L=\relax]{1};
  \coordinate (2) at (3.00,0.000);
  \coordinate (3) at (0.000,0.000);
  \BlueEdges
  \LoopE(1)
  \LoopW(1)
\GraphEdgeSignS(1)(3){0.5}{n}\GraphEdgeSignS(1)(2){0.5}{n}\ESwapOfs1312{in=160,out=20}{0.2,0.3}{-0.2,0.3}{0.5}\end{tikzpicture}}}

\def\tgrGH{\raise-7pt\hbox{\begin{tikzpicture}[scale=\GraphScale]
  \GraphVertices
  \Vertex[x=1.50,y=0.000,L=\relax]{1};
  \coordinate (2) at (3.00,0.000);
  \coordinate (3) at (0.000,0.000);
  \BlueEdges
  \LoopE(1)
  \LoopW(1)
\GraphEdgeSignS(1)(3){0.5}{n}\GraphEdgeSignS(1)(2){0.5}{n}\EArrOfs1312{in=150,out=30}{0.1,0.29}{0,0.35}{0.5}\EArrOfs1213{in=-60,out=-60}{0,0.2}{0.3,-0.25}{0.5}\end{tikzpicture}}}

\def\tgrGA{\raise-3pt\hbox{\begin{tikzpicture}[scale=\GraphScale]
  \GraphVertices
  \Vertex[x=0.000,y=0.000,L=2]{1};
  \BlueEdges
\end{tikzpicture}}}

\title{The 2-parity conjecture for elliptic curves with isomorphic 2-torsion}\author{Holly Green, C\'eline Maistret}\date{}

\address{Department of Mathematics, University College London, London WC1H 0AY, UK}
\email{h.green.19@ucl.ac.uk}

\address{School of Mathematics, University of Bristol, Bristol BS8 1UG, UK}
\email{celine.maistret@bristol.ac.uk}

\begin{document}

\begin{abstract}
The Birch and Swinnerton--Dyer conjecture famously predicts that the rank of an elliptic curve can be computed from its $L$-function. In this article we consider a weaker version of this conjecture called the parity conjecture and prove the following. Let $E_1$ and $E_2$ be two elliptic curves defined over a number field $K$ whose 2-torsion groups are isomorphic as Galois modules. 
Assuming finiteness of  
the Shafarevich-Tate groups of $E_1$ and $E_2$, we show that the Birch and Swinnerton-Dyer conjecture correctly predicts the parity of the rank of $E_1\times E_2$. Using this result, we complete the proof of the $p$-parity conjecture for elliptic curves over totally real fields. 
\end{abstract}

\maketitle

{\hypersetup{}
\setcounter{tocdepth}{1}
\tableofcontents}

\section{Introduction}
\subsection{Main results}\label{s:mainresults}
For an abelian variety $A$ over a number field $K$, the Birch--Swinnerton-Dyer conjecture predicts that the algebraic rank of $A/K$ is given by the order of vanishing of its $L$-function $L(A/K,s)$ at $s=1$. In addition, the completed $L$-function of $A/K$ is expected to satisfy a functional equation centered around $s=1$, whose sign is given by the global root number of $A/K$. It then follows that the global root number should predict the parity of the order of vanishing of $L(A/K,s)$ at $s=1$, which combined with the Birch--Swinnerton-Dyer yields the parity conjecture. 
\begin{conjecture}[Parity conjecture] For every abelian variety $A$ over a number field $K$, 
$$
(-1)^{\rk(A/K)} = w_{A/K},
$$
where $\rk(A/K)$ and $w_{A/K}$ denote the rank and global root number of $A/K$ respectively.
\end{conjecture}
The global root number is defined as a product of local root numbers which can be computed via the local Galois representations of $A/K$. The parity conjecture therefore offers an effective way to compute the parity of $\rk(A/K)$. 
\newpage
There is no known successful approach to prove the parity conjecture unconditionally at present. However, assuming the finiteness of the Shafarevich-Tate group of $A/K$, several instances of the parity conjecture have been proven via the $p$-parity conjecture. Recall that, for each prime $p$, $\Sha_{A/K}[p^{\infty}]$ decomposes as $(\Q_p/\Z_p)^{\delta_p} \times \Delta$, where $\Delta$ is a finite group and $\delta_p$ is a non-negative integer. \\Define the $p^{\infty}$-Selmer rank as $\rk_p(A/K) = \rk(A/K) + \delta_p$. If $\Sha_{A/K}$ is indeed finite then $\delta_p =0$ for all $p$ so that $\rk_p(A/K) = \rk(A/K)$ for all $p$. This prompts the $p$-parity conjecture.
\begin{conjecture}[$p$-Parity conjecture] For every abelian variety $A$ over a number field $K$ and every prime $p$,
$$
(-1)^{\rk_p(A/K)} = w_{A/K}.
$$
\end{conjecture}

In certain settings, the parity of the $p^{\infty}$-Selmer rank is computable and this can lead to a proof of the $p$-parity conjecture. Most notably the $p$-parity conjecture is known for elliptic curves over $\Q$ (\cite{DDBSDmodulosquare}), for elliptic curves over number fields admitting a $p$-isogeny (\cite{kestutis}), for elliptic curves over totally real number fields when $p \ne 2$ (and in all non CM cases and some CM cases when $p=2$) (\cite{dokchitser2011root},\cite{Nekovar},\cite{Nekovar1},\cite{Nekovar2},\cite{Nekovar3}) and for quadratic twists of elliptic curves (\cite{KramerTunnell}). In particular the $p$-parity conjecture is still open for elliptic curves over number fields in general.  In higher dimensions, the $p$-parity conjecture is known for principally polarized abelian surfaces subject to local conditions (\cite{parityforab}), for Jacobians of hyperelliptic curves base-changed from a subfield of index 2 (\cite{Morgan}) and abelian varieties admitting a suitable isogeny (\cite{Coates}).
In this paper, we consider two elliptic curves whose 2-torsion groups are isomorphic as Galois modules and prove the following. 
\begin{theorem}[Theorem \ref{thm:2tors}]\label{maintheorem}
Let $E_1, E_2/K$ be elliptic curves over a number field. If $E_1[2]\cong E_2[2]$ as Galois modules, then the $2$-parity conjecture holds for $E_1/K$ if and only if it holds for $E_2/K$.
\end{theorem}

Combining Theorem \ref{maintheorem}  with the known results for the 2-parity conjecture for non-CM elliptic curves over totally real fields mentioned above, we prove the 2-parity conjecture for all CM curves over totally real fields.  

\begin{theorem}[Theorem \ref{2PCtotallyreal}] Let $E/K$ be an elliptic curve over a totally real number field with complex multiplication. The $2$-parity conjecture holds for $E/K$. \end{theorem}
In turn, we complete the proof of the $p$-parity conjecture for elliptic curves over totally real fields. 
\begin{theorem}[Corollary \ref{maincorollary}]
Let $p$ be a prime and $K$ be a totally real field. The $p$-parity conjecture holds for all elliptic curves over $K$.
\end{theorem}

\subsection{Setup}
The condition on the 2-torsion groups in Theorem \ref{maintheorem} provides us with a simple setup which we will use throughout this article.  

\begin{definition}\label{EE'}
Let $f(x)$ be a separable monic cubic polynomial such that $f(0) \ne 0$. We let 
$$
E: y^2 = f(x), \quad E': y^2 = xf(x).
$$ 
\end{definition}

Thanks to Lemma \ref{lem:2tors}, it is enough to work with $E$ and $E'$ to recover Theorem \ref{maintheorem} (see Corollary 1.6 in \cite{dokchitser2011root} for results on $2$-parity conjecture for elliptic curves over quadratic extensions). We therefore prove the following.

\begin{theorem}[Theorem \ref{thm:2PCforEC}]\label{secondmain}Let $K$ be a number field. The $2$-parity conjecture holds for $E/K$ if and only if it holds for $\Jac E'/K$.\end{theorem}

Lemma \ref{lem:isogeny} also guarantees the existence of an isogeny $\phi : E \times \Jac E' \to \Jac C$, where $C:y^2=f(x^2)$. It follows that Theorem \ref{secondmain} yields
 \begin{theorem}[Theorem \ref{thm:2PCsplit}]\label{thirdmain} Let $K$ be a number field, $c\in K^\times$ and $f(x)\in K[x]$ be a separable monic cubic polynomial such that $f(0) \ne 0$. The $2$-parity conjecture holds for $\Jac C/K$ where $C: y^2 = cf(x^2)$ is a genus $2$ hyperelliptic curve.
\end{theorem}

\begin{remark}Since no condition is imposed on the Galois group of the polynomial defining $C$, Theorem \ref{thirdmain} generalises Theorem 11.2 in \cite{parityforab} in the case of Jacobians of genus 2 curves which are isogenous to a product of elliptic curves. 
\end{remark}

As mentioned before, assuming finiteness of the Shafarevich-Tate group allows us to convert our results on the $2$-parity conjecture into results on the parity conjecture. 

\begin{corollary}
Let $K$ be a number field.

\begin{enumerate}
\item If $E_1[2]\cong E_2[2]$ as Galois modules and both $\Sha_{E_1/K}[2^{\infty}]$ and $\Sha_{E_2/K}[2^{\infty}]$ are finite then the parity conjecture holds for $E_1/K$ if and only if it holds for $E_2/K$.
\item If $\Sha_{E/K}[2^{\infty}]$ and $\Sha_{\Jac E'/K}[2^{\infty}]$ are finite then the parity conjecture holds for $E/K$ if and only if it holds for $\Jac E'/K$.
\item If $\Sha_{\Jac C/K}[2^{\infty}]$ is finite then the parity conjecture holds for $\Jac C/K$.

\end{enumerate}

\end{corollary}

\subsection{Approach}

Our approach relies on the existence of $\phi $ (Lemma \ref{lem:isogeny}) which we use to express the parity of the $2^{\infty}$-Selmer rank of $E \times \Jac E'$ as a product of local terms running over all places of $K$. 
\begin{definition}\label{de:localterm} Let $\mathcal K$ be a local field 
 . We write
$$
\lambda_{f,\mathcal K} = \mu_{C/\mathcal K}\cdot (-1)^{\dim_{\F_2}\ker\phi|_{\mathcal K} - \dim_{\F_2}\coker\phi|_{\mathcal K}},
$$
where $\mu_{C/ \mathcal K}$ is $-1$ if $C$ is deficient over $\mathcal K$ (as defined in \cite[\S 8]{Poonen}) and $1$ otherwise.\\
Here $\ker\phi|_{\mathcal K} = (E\times\Jac E')(\mathcal K)[\phi]$ and $\coker\phi|_{\mathcal K} = |\Jac C(\mathcal K)/\phi((E\times\Jac E')(\mathcal K))|$.
\end{definition}

\begin{theorem}[Theorem \ref{thm:paritythm}]
Let $K$ be a number field.  
Then
$$
  (-1)^{\rk_2 E/K+\rk_2 \Jac E'/K}= \prod_v \lambda_{f,K_v},
$$
where the product runs over all places of $K$.
\end{theorem}

Using this expression for the parity of the $2^{\infty}$-Selmer rank of $E \times \Jac E'$, we prove that the $2$-parity conjecture holds for $E \times \Jac E'$ by comparing $\lambda_{f,K_v}$ and $w_{E/K_v}w_{\Jac E'/K_v}$ place by place and showing that 
$$
\prod_v \lambda_{f,K_v} = \prod_v w_{E/K_v}w_{\Jac E'/K_v}, 
$$
in turn proving that 
$$
(-1)^{\rk_2 E/K+\rk_2 \Jac E'/K} = w_{E/K}w_{\Jac E'/K}.
$$

The point of this approach is that one can express each $\lambda_{f,K_v}$ in terms of local arithmetic invariants of $E, \Jac E'$ and $\Jac C$ (Lemmas \ref{lem:kercoker} and \ref{le:comps}). In most cases (when $v$ is an infinite or an odd place with semistable reduction), $\lambda_{f,K_v}$ and the local root numbers $w_{E/K_v}w_{\Jac E'/K_v}$ can be computed and compared locally. This leads to the most tricky part of the proof. It turns out that in general, $\lambda_{f,K_v}$ and $w_{E/K_v}w_{\Jac E'/K_v}$ do not agree locally. The bulk of the proof is then to find an expression $H_{f,K_v}$ for that local discrepancy satisfying the product formula $\prod_v  H_{f,K_v} =1$ and prove that
\begin{equation}\label{localequality}\tag{$\dagger$}
w_{E/K_v}w_{\Jac E'/K_v}=\lambda_{f,K_v}H_{f,K_v} \quad \text{ for all } v.
\end{equation}
In this article we propose an expression for $H_{f,K_v}$ in the form of a product of Hilbert symbols (Definition \ref{de:errorterm}).
 
Finding the local error term $H_{f,K_v}$ is a common challenge in all instances of the proof of the $p$-parity conjecture listed in Section \ref{s:mainresults}. This is mainly due to the fact that no conceptual understanding of this local discrepancy exists at the moment. In particular, so far, each proof has exhibited a different ad hoc expression for the local error term with no obvious link to one another.

In Appendix \ref{gener} the first author conjectures an expression for a generalisation of the error term defined in this paper for certain hyperelliptic curves, and observe that in the case of elliptic curves, it reduces to the error term presented in \cite{DDparity}. The appendix also explains what prompted our formulation of $H_{f,K_v}$. 

\subsection{Overview}

In Section \ref{s:isogenyparity} we establish the existence of $\phi$, prove a few of its properties, write $(-1)^{\rk_2 E/K+\rk_2 \Jac E'/K}$ as a product of local invariants (Definition \ref{de:localterm}, Lemmas \ref{lem:kercoker} and \ref{le:comps}), and introduce the local discrepancy $H_{f,K_v}$ (Definition \ref{de:errorterm}). Section \ref{s:Archthmproof} proceeds to prove that 
(\ref{localequality}) holds over archimedean fields. This is followed by Section \ref{s:nonArchthmproof} which presents the theory of clusters that are used to prove (\ref{localequality}) over non-archimedean local fields when $f(x)$ is integral, both $E$ and $E'$ have semistable reductions and $xf(x)$ has at most one double root. Section \ref{s:thmproof} then deals with the remaining cases by using a continuity argument in Section \ref{s:continuity}, a few reduction steps and then a global-local argument in Section \ref{s:reduction}. Lastly, Section \ref{s:global} establishes our global results by taking the product of (\ref{localequality}) over all places of $K$ and using the product formula for Hilbert symbols. 

\subsection{Related work}
We note that in a related work presented in \cite{AAS}, the authors compare the $p$-Selmer
rank, global and local root numbers of two elliptic curves with equivalent mod $p$ Galois representations over number fields when the prime $p$ is \emph{odd}.

\subsection{Acknowledgements}
We would like to warmly thank Vladimir Dokchitser and Adam Morgan for helpful discussions and suggestions. We also thank the anonymous referee for suggesting to look into proving Theorem \ref{2PCtotallyreal}. The first author is supported by University College London and the EPSRC studentship grant EP/R513143/1. The second author is supported by a Royal Society Dorothy Hodgkin Fellowship and was supported by the Simons Collaboration on Arithmetic Geometry, Number Theory, and Computation (Simons Foundation grant number 550023) during part of this work.

\section{Parity of the $2^\infty$-Selmer rank }\label{s:isogenyparity}

In this section we present how to control the parity of the $2^\infty$-Selmer rank of $E\times \Jac E'$. The results here mimic those of \cite[\S 3]{parityforab}.

\begin{notation}
Let $K$ be a field of characteristic $0$ and let $f(x)\in K[x]$ be a separable monic cubic polynomial with $f(0) \ne 0$. We write
$$
f(x) = (x-\alpha_1)(x-\alpha_2)(x-\alpha_3).
$$
\end{notation}

\begin{remark}\label{re:Jacobian} Since $E':y^2 = xf(x)$ has a rational point, the map $X = -\frac{\alpha_1\alpha_2\alpha_3}{x}$, $Y=\frac{\alpha_1\alpha_2\alpha_3y}{x^2}$ gives the following model for $\Jac E'$. 
$$
\Jac E':y^2 = (x+\alpha_2\alpha_3)(x+\alpha_1\alpha_3)(x+\alpha_1\alpha_2).$$

\end{remark}

\begin{lemma}\label{lem:isogeny} 

Let $K$ be a field of characteristic $0$ and let $f(x)\in K[x]$ be a separable monic cubic polynomial with $f(0)\ne 0$. There is an isogeny $\phi : E\times\Jac E' \rightarrow \Jac C$, where $C: y^2 = f(x^2)$, such that the dual isogeny $\phi^t$ satisfies $\phi\phi^t = [2]$, with 
$$\ker\phi = \{O, \big ((\alpha_i,0), (-\frac{\alpha_1\alpha_2\alpha_3}{\alpha_i},0)\big )\text{ for } i = 1,2,3\}.$$

\end{lemma}

\begin{proof}
This is Theorem 14.1.1 in \cite{CasselsFlynn}, where to recover $C:y^2=f(x^2)$ we let $G_i = (x^2-\alpha_i)$ in (ii) for $i=1,2,3$.
In our case the maps $C\to E$ and $C\to \Jac E'$ are given in (14.1.4) and (14.1.5) respectively by
$$
(x,y) \mapsto (x^2, y) \text{ and } (x,y) \mapsto (\frac{-\alpha_1\alpha_2\alpha_3}{x^2}, \frac{\alpha_1\alpha_2\alpha_3}{x^3}).
$$
Since $\ker \phi =\phi^t(\Jac C[2])$ we use the explicit description of the isogeny $\phi^t:\Jac C \to E \times \Jac E'$ given in the footnote of Theorem 14.1.1 to compute $\phi^t(\Jac C[2])$. 

Fix a choice of square root for $\alpha_2$ and $\alpha_3$ and consider $P \in \Jac C[2]$, given by $P = \{(\sqrt{\alpha_2},0), (\sqrt{\alpha_3},0)\}$. Then 
$$
\phi^t(P) = \big ((\alpha_2,0) + (\alpha_3,0) , (-\alpha_1\alpha_3,0) + (-\alpha_1\alpha_2, 0)\big ) =\big( (\alpha_1, 0), (-\alpha_2\alpha_3,0)\big ) \in E \times \Jac E'.$$
The result follows by permuting indices. \end{proof}

\subsection{Parity theorem} The isogeny $\phi$ allows us to express the parity of the $2^{\infty}$-Selmer rank of $E ~\times~ \Jac E'$ as a product of local terms. 

\begin{theorem}\label{thm:paritythm}
Let $K$ be a number field and let $f(x)\in K[x]$ be a separable monic cubic polynomial with $f(0) \ne 0$. Then
$$
  (-1)^{\rk_2 E/K+\rk_2 \Jac E'/K}= \prod_v \lambda_{f,K_v},
$$
where $\lambda_{f,K_v}$ is as Definition \ref{de:localterm} and the product runs over all places of $K$.
\end{theorem}
\begin{proof}
This follows from \cite{parityforab} Theorem 3.2 with $A =E \times \Jac E'$ and $A' = \Jac C$. 
\end{proof}

\subsection{Kernel/Cokernel on local points}

The local arithmetic data attached to the curves $E, E'$ and $C$ provide us with a practical way to compute the quantity $\lambda_{f,K_v}$.

\begin{notation}
For a curve $\mathcal C/\R$ with Jacobian $A$, we write $n_{\mathcal C/\R}$ for the number of connected components of $\mathcal C(\R)$. We write $A(\R)^\circ$ for the connected component of $A$ containing the identity and $n_{A/\R}~=~|A(\R)/A(\R)^\circ|$ for the number of connected components. 
\end{notation}

For $\mathcal K/\Q_p$ a finite extension, we denote by $c_{A/\mathcal K}$ and $ \omega_{A/\mathcal K}^\circ$ the Tamagawa number and N\'eron exterior form for $A/\mathcal K$.

\begin{lemma}[\cite{parityforab} Lemma 3.4]\label{lem:kercoker} Let $\mathcal K$ be a local field and let $f(x)\in \mathcal K[x]$ be a separable monic cubic polynomial with $f(0) \ne 0$.
Let $\phi$ be as in Lemma \ref{lem:isogeny}. Then
$$
\frac{|\ker\phi|_{\mathcal K}|}{|\mathrm{coker}\phi|_{\mathcal K}|} = \left\{
\begin{array}{cl}
  4 & \text{if }\mathcal K\simeq\C,\cr
 \bigl|(E\times\Jac E')(\mathcal K)^\circ[\phi]\bigr|\cdot  {n_{E/\mathcal K}n_{\Jac E'/\mathcal K}}/{n_{\Jac C/\mathcal K}}& \text{if }\mathcal K\simeq\R, \cr
 c_{E/\mathcal K}c_{\Jac E'/\mathcal K}/c_{\Jac C/\mathcal K} & \text{if }\mathcal K/\Q_p \text{ finite, $p$ odd},\cr
 c_{E/\mathcal K}c_{\Jac E'/\mathcal K}/{c_{\Jac C/\mathcal K}}\cdot \Bigl|\frac{\phi^*\omega^\circ_{\Jac C/\mathcal K}}{\omega^\circ_{(E\times\Jac E')/\mathcal K}}\Bigr|_{\mathcal K} & \text{if } \mathcal K/\Q_2 \text{ finite},
\end{array}
\right.
$$ where $|x|_{\mathcal K}$ denotes the normalised absolute value of $x$.
\end{lemma}

\begin{lemma}[\cite{Gross} Prop. 3.2.2 and 3.3]
\label{le:comps}
For a smooth projective curve $\mathcal C$ over $\R$
$$
 n_{\Jac \mathcal C/\R} = \left\{
   \begin{array}{ll}
       2^{n_{\mathcal C/\R}-1} & {\text{if }} n_{\mathcal C/\R} > 0,\cr
       1 & {\text{if }} n_{\mathcal C/\R} = 0 {\text{ and }} \mathcal C {\text{ has even genus}},\cr
       2 & {\text{if }} n_{\mathcal C/\R} = 0 {\text{ and }}\mathcal C {\text{ has odd genus}}.
   \end{array}
\right.
$$
\end{lemma}

\subsection{The local discrepancy} We present here an expression for the local discrepancy between $\lambda_{f,\mathcal K}$ and $w_{E/\mathcal K}w_{\Jac E'/ \mathcal K}$.

\begin{definition}\label{de:errorterm} Let $\mathcal K$ be a local field of characteristic 0 and let $f(x) = x^3+ax^2+bx+c\in \mathcal K[x]$ be such that $c\neq0$ and $ \Delta_f:= 18abc-4a^3c+a^2b^2-4b^3-27c^2\neq0$. Let $L=ab-9c$ and define $\E$ to be the following product of Hilbert symbols
$$\E = \begin{cases}(b,-c)_{\mathcal K}\cdot(-2L,\Delta_f)_{\mathcal K}\cdot(L,-b)_{\mathcal K}& b,L\neq0,\\(-c,-1)_{\mathcal K}\cdot(2c,\Delta_f)_{\mathcal K}& \text{otherwise}.\end{cases}$$ \end{definition}

\begin{remark}As $b$ or $L$ approach 0, both expressions for $H_{f,\mathcal K}$ agree. This can be seen in the proof of Lemma \ref{lem:thmcontinuity}.\end{remark}

      \begin{remark}\label{re:invariants}The invariant $L$ appearing in the definition of $\E$ can be written in terms of the roots of $f$ as
    $$L = 8\alpha_1\alpha_2\alpha_3-(\alpha_1+\alpha_2)(\alpha_1+\alpha_3)(\alpha_2+\alpha_3).$$      \end{remark}

\begin{theorem}[Local theorem]\label{thm:localthm} Let $\mathcal K$ be a local field of characteristic 0 and let $f(x)\in \mathcal K[x]$ be a separable monic cubic polynomial with $f(0) \neq0$. Then
$$w_{E/\mathcal K}w_{\Jac E'/\mathcal K} = \lambda_{f, \mathcal K} \E.$$
\end{theorem}

The essential property of this description is that it explains why $w_{E/\mathcal K}w_{\Jac E'/\mathcal K}$ and $ \lambda_{f,\mathcal K}$ always differ at an even number of places, since, by the product formula for Hilbert symbols, $\prod_v H_{f,K_v}\!=\!1$, where $K$ is a number field and the product runs over all places of $K$.
Sections \ref{s:Archthmproof}, \ref{s:nonArchthmproof} and \ref{s:thmproof} will be dedicated to proving this statement.

\section{Local Theorem over archimedean fields}\label{s:Archthmproof}

Here we prove that Theorem \ref{thm:localthm} holds when $\mathcal K$ is an archimedean local field.

\begin{proposition}\label{thm:localcomplex} Theorem \ref{thm:localthm} holds when $\mathcal K = \C$.\end{proposition}

\begin{proof}Here $w_{E/\C}=w_{\Jac E'/\C}=-1$, $\Ec = 1$ and by Lemma \ref{lem:kercoker} $\lambda_{f,\C} = (-1)^2=1$.\end{proof}

\begin{proposition}\label{thm:localreals} Theorem \ref{thm:localthm} holds when  $\mathcal K = \R$.\end{proposition}

 \begin{proof} 
 
We note first that $w_{E/\R}w_{\Jac E'/\R} = +1$ so we need only verify that $\lambda_{f,\R}= \Er$. Table $\ref{table:reals}$ gives the values of $n_{E/\R}, n_{\Jac E'/\R}, n_{\Jac C/\R}, \bigl|(E\times\Jac E')(\R)^\circ[\varphi]\bigr|,\mu_{C/\R},\lambda_{f,\R}$ and $\Er$ for each possible arrangement of the real roots of $xf(x)$. In particular, column 2 lists the real roots of $xf(x)$ from smallest to largest where the roots of $f$ are denoted by red circles
($\smash{\raise4pt\hbox{\clusterpicture\Root[D]{1}{first}{r1};\endclusterpicture}}$), and the root 0 (of $x$) by a blue diamond ($\smash{\raise4pt\hbox{\clusterpicture\Root[C]{1}{first}{r1};\endclusterpicture}}$). 
 
  \begin{table}[h!]
\def\cltopskip{0pt}              
\def\clbottomskip{3pt}           
\def\graphdslabelscale{1.0}      
\def\cdepthscale{0.6}            
$$
\scalebox{1}{$
\begin{array}{| c|c||c|c|c|c|c|c|c|}
\hline
\text{Case} & \text{Real roots}  & n_{E/\R} & n_{\Jac E'/\R} & n_{\Jac C/\R}&  \bigl|(E\times\Jac E')(\R)^\circ[\varphi]\bigr| &\mu_{C/\R}&\lambda_{f,\R}& \Er \\
\hline
1&{\smash{\raise2pt\hbox{\clusterpicture\Root[C] {1} {first} {r1};\Root[D] {} {r1} {r2};\Root[D] {} {r2} {r3};\Root[D] {} {r3} {r4};\endclusterpicture}}}  & 2& 2 & 4 &2 & 1&-1& -1\\
 \hline
2 & {\smash{\raise2pt\hbox{\clusterpicture\Root[D] {1} {first} {r1};\Root[C] {} {r1} {r2};\Root[D] {} {r2} {r3};\Root[D] {} {r3} {r4};\endclusterpicture}}} &2& 2 & 2  &1& 1&-1& -1 \\
 \hline
3 &{\smash{\raise2pt\hbox{\clusterpicture\Root[D] {1} {first} {r1};\Root[D] {} {r1} {r2};\Root[C] {} {r2} {r3};\Root[D] {} {r3} {r4};\endclusterpicture}}} &2& 2 &1& 1  &1 &+1 & +1\\
 \hline
4  &  {\smash{\raise2pt\hbox{\clusterpicture\Root[D] {1} {first} {r1};\Root[D] {} {r1} {r2};\Root[D] {} {r2} {r3};\Root[C] {} {r3} {r4};\endclusterpicture}}} &2& 2 & 1&1  &1 &+1 & +1\\
 \hline
5  & {\smash{\raise2pt\hbox{\clusterpicture\Root[C] {1} {first} {r1};\Root[D] {} {r1} {r2};\endclusterpicture}}} &1& 1&1 & 2  &1 &-1& -1 \\
 \hline
6  & {\smash{\raise2pt\hbox{\clusterpicture\Root[D] {1} {first} {r1};\Root[C] {} {r1} {r2};\endclusterpicture}}}  &1& 1 &1& 2  &1 &-1 & -1\\
\hline
\end{array}$}$$\caption{Local data when $\K=\R$}\label{table:reals}\end{table}

The contents of columns 3, 4 and 5 are determined using Lemma \ref{le:comps} and observing that \\
$n_{E/\R}~ = ~n_{E'/\R}~ = ~2$ when $f$ has 3 real roots and 1 otherwise, and that $n_{C/\R} = 3$ when $f$ has 3 positive real roots, $2$ when $f$ has 2 positive real roots and 1 otherwise.

For column 6 we use the description of $\ker\phi$ as given in Lemma \ref{lem:isogeny}. We count how many elements belong to the identity component of $(E\times\Jac E')(\R)$ (i.e. each entry of the pair lies on the identity component of the corresponding curve).
Clearly 
the point at infinity on $(E\times\Jac E')$ always satisfies these conditions. Let $\alpha_3$ be the largest real root of $f$, so that $(\alpha_3,0)\in E(\R)^\circ$ and $(\alpha_1,0), (\alpha_2,0)~\notin ~E(\R)^\circ$. Now $\bigl|(E\times\Jac E')(\R)^\circ[\varphi]\bigr|$ is $2$ precisely when $(-\alpha_1\alpha_2,0)\in\Jac E'(\R)^\circ$, i.e. when $-\alpha_1\alpha_2$ is the largest real element of $S = \{-\alpha_1\alpha_2, -\alpha_2\alpha_3, -\alpha_1\alpha_3\}$, and $1$ otherwise. In cases 1-4, $\alpha_1, \alpha_2, \alpha_3\in\R$ so $S\subseteq\R$ and it can be observed that $-\alpha_1\alpha_2$ is the largest precisely when $\alpha_1, \alpha_2>0.$ In cases 5 and 6, $\alpha_1 = \bar\alpha_2$ and so $-\alpha_1\alpha_2$ is the only real element of $S$, in particular it is the largest real element.

Column 7 keeps track of the deficiency contribution from $C/\R$. As $C(\R)$ is always non-empty, this is 1.

Column 8 gives the value of $\lambda_{f,\R}$ calculated using Lemma \ref{lem:kercoker}.

We now justify the value of $\Er$, as given in column 9, via a case by case analysis of the signs of the Hilbert symbol entries.

\textbf{Claim:} When $b,L\neq 0$, $\Er = (-2L, \Delta_f)_\R\cdot(L,-ac)_\R\cdot(a,-c)_\R$. {This is true as $(L, abc)_\R\!=\!(ab, -c)_\R$ by the standard Hilbert symbol identity $(A+B,-AB)_\R\!=\!(A,B)_\R$, thus $(L,-b)_\R\!=\!(L, -ac)_\R \cdot(ab, -c)_\R$.}

Since the sign of $a$ is easier to control than that of $b$, we will use this equivalent expression for $\Er$ whenever $b,L\neq0$.

1,4) Here $\Delta_f, b, ac>0$. Applying the AM-GM inequality gives that $\alpha_i+\alpha_j>2\sqrt{\alpha_i\alpha_j}$ in case 1 and $\alpha_i+\alpha_j>-2\sqrt{\alpha_i\alpha_j}$ in case 4, for $i\neq j$. In particular, in case 1 we have $L<0$ and $\Er = -1$ and in case 4 we have $L>0$ and $\Er = 1$.

2) Here $\Delta_f, c>0$. If $b,L\neq 0$ then $\Er =(L, -a)_\R\cdot(a, -1)_\R$ which is $-1$ unless $a,L>0$. Suppose $\alpha_1<0$, if $a>0$ then $\alpha_1 + \alpha_2, \alpha_1 +\alpha_3<0$ and $\alpha_2+\alpha_3>0$ so $L<0$ using Remark \ref{re:invariants}. If $bL=0$ then clearly $\Er$ again evaluates to $-1$.

3) Here $\Delta_f>0, c<0$. If $b,L\neq 0$ then $\Er = (L,a)_\R$ which is $1$ unless $a,L<0$. Suppose $\alpha_1,\alpha_2<0$, if $a<0$ then $\alpha_1 + \alpha_3, \alpha_2 +\alpha_3>0$ and $\alpha_1+\alpha_2<0$ so $L>0$ using Remark \ref{re:invariants}. If $bL=0$ then clearly $\Er$ again evaluates to $1$.

5) Here $\Delta_f,c<0$. If $b,L\neq 0$ then $\Er =-(L,-a)_\R$ which is $-1$ unless $a>0$, $L<0$. Suppose $\alpha_2 = \bar\alpha_1$, if $a>0$ then $\alpha_1 + \alpha_2<0$ so $L>0$ using Remark \ref{re:invariants}. If $bL=0$ then clearly $\Er$ again evaluates to $-1$.

6) Here $\Delta_f<0, c>0$. If $b,L\neq 0$ then $\Er =-(-L,a)_\R$ which is $-1$ unless $a<0$, $L>0$. Suppose $\alpha_2 = \bar\alpha_1$, if $a<0$ then $\alpha_1 + \alpha_2>0$ so $L<0$ using Remark \ref{re:invariants}. If $bL=0$ then clearly $\Er$ again evaluates to $-1$. \end{proof}

\section{Local Theorem over non-archimedean fields for nice reduction types}\label{s:nonArchthmproof}

Here we prove that Theorem \ref{thm:localthm} holds when $\mathcal K/\Q_p$ is a finite extension for an odd prime $p$ and the reduction of $E'$ is of a certain type. 
We write $\pi_{\K}$ for a choice of uniformiser of $\K$ and $v$ for the normalised valuation. 
We begin by recalling some useful terminology. 

\subsection{Clusters: curves over local fields with odd residue characteristic}\label{ss:clusterpics}

To keep track of the arithmetic invariants needed to compute $\lambda_{f,\K}$ as in Lemma \ref{lem:kercoker}, we use the machinery of ``clusters'' (\cite{m2d2}, \cite{usersguide}). Clusters allow us to extract arithmetic invariants of hyperelliptic curves over $p$-adic fields with $p$ odd from simple combinatorial data (see Example \ref{examplecluster} below for a worked out example).

\begin{definition}[Clusters]
\label{def:cluster}
Let $\mathcal K$ be a finite extension of $\Q_p$ where $p$ is an odd prime 
and $C\!:\!y^2\!=\!cf(x)$ a hyperelliptic curve, where $c\in \mathcal K^\times$ and $f(x)\in \mathcal K[x]$ is monic with set of roots $\cR$. A {\em cluster} is a non-empty subset $\s\subset\cR$ of the form $\s = D \cap \cR$ for some disc $D=\{x\!\in\! \bar{\mathcal K} \mid v(x-z)\!\geq\! d\}$ for some $z\in \bar{\mathcal K}$~and~$d\in \Q$. 

For a cluster $\s$ of size $>1$, its {\em depth} $d_\s$ is the maximal $d$ for which $\s$ is cut out by such a disc, that is $d_\s\! =\! \min_{r,r' \in \mathfrak{s}} v(r\!-\!r')$. If moreover $\s\neq \cR$, then its {\em relative depth} is $\delta_\s\! =\! d_\s\! -\!d_{P(\s)}$, where $P(\s)$ is the smallest cluster with $\s\subsetneq P(\s)$ (the ``parent'' cluster).

We refer to this data as the {\em cluster picture} of $C$, denoted $\Sigma_{C/\mathcal K}$.

\end{definition}

\begin{notation}\label{not:drawcluster}
We draw cluster pictures by drawing roots $r \in \cR$ as \smash{\raise4pt\hbox{\clusterpicture\Root[]{1}{first}{r1};\endclusterpicture}} and drawing ovals around roots to represent clusters (of size $>1$). If we wish to specify which root is which, we instead draw the roots of $f$ as red circles
$\smash{\raise4pt\hbox{\clusterpicture\Root[D]{1}{first}{r1};\endclusterpicture}}$ and the root 0 (of $x$) as a blue diamond $\smash{\raise4pt\hbox{\clusterpicture\Root[C]{1}{first}{r1};\endclusterpicture}}$. For example:
$$
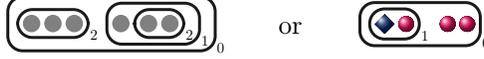
\scalebox{1.4}{\clusterpicture            
  \Root[] {1} {first} {r1};
  \Root[] {} {r1} {r2};
  \Root[] {} {r2} {r3};
  \Root[] {3} {r3} {r4};
  \Root[] {1} {r4} {r5};
  \Root[] {} {r5} {r6};
  \ClusterD c1[{2}] = (r1)(r2)(r3);
  \ClusterD c2[{2}] = (r5)(r6);
  \ClusterD c3[{1}] = (r4)(c2);
  \ClusterD c4[{0}] = (c1)(c2)(c3);
\endclusterpicture}
\qquad
\text{or}
\qquad
\scalebox{1.4}{\clusterpicture            
  \Root[C] {1} {first} {r2};
  \Root[D] {} {r2} {r3};
  \ClusterD c1[1] = (r2)(r3);
  \Root[D] {1} {c1} {r4};
  \Root[D] {} {r4} {r5};
  \ClusterD c3[0] = (r2)(c1)(r4)(r5);
\endclusterpicture}
$$
The subscript on the largest cluster $\cR$ is its depth; on the other clusters it is their relative depth.
\end{notation}

\begin{definition}[Twins and signs]\label{def:tangenttwin}
A {\em twin} is a cluster of size 2. For each twin $\mathfrak t = \{r_1, r_2\}$ pick a square root $\theta_{\mathfrak t}=\sqrt{\prod_{r\notin\mathfrak t}(\frac{r_1+r_2}{2}-r)}$ and define its {\em sign} $\pm$ through the formula $\frac{\Frob_{\mathcal K}(\theta_{\mathfrak t})}{\theta_{\Frob_{\mathcal K}(\mathfrak t)}}\equiv\pm1$ in the residue field, where $\Frob_{\mathcal K}$ denotes a Frobenius element in $\Gal(\bar{\mathcal K}/\mathcal K)$. Note that the sign depends on the choice of square root. \end{definition}

\begin{example}\label{examplecluster} Let $\K = \Q_{17}$ and $E: y^2 = (x-17)(x-1)(x-2)$. Then $E':y^2=x(x-17)(x-1)(x-2)$ and by Remark \ref{re:Jacobian} $\Jac E': y^2 = (x+2)(x+34)(x+17)$. We see that both $\Sigma_{E'/\Q_{17}}$ and $\Sigma_{\Jac E'/\Q_{17}}$ have a twin cluster $\mathfrak t_1=\{0,17\}$ and $\mathfrak t_2=\{-34,-17\}$ respectively. To determine the sign on $\mathfrak t_1$ we compute $\theta_{\mathfrak t_1} = \sqrt{(\frac{17}{2}-1)(\frac{17}{2}-2)}$ 
and find that $\frac{\Frob_{\Q_{17}}(\theta_{\mathfrak t_1})}{\theta_{\Frob_{\Q_{17}}(\mathfrak t_1)}}\equiv 1 \mod 17$. It follows that both $\mathfrak t_1$ and $\mathfrak t_2$ have sign $+1$. Therefore
$$
\Sigma_{E/\Q_{17}} = \goodE, \quad \Sigma_{E'/\Q_{17}} = \ctEprime, \quad \Sigma_{\Jac E'/\Q_{17}} = \JCtwin.
$$
Using \cite[Theorems 3.1 \& 3.3]{omri} with $f_1(x) = x$, $f_2(x) =  (x-17)(x-1)(x-2)$  we find that  $\Upsilon_{C/\Q_{17}}$ is $\tgrGBex$ where Frobenius fixes the loop. Following \cite{m2d2} Table 1.1, $C$ has type  $1_1^+$  so that we have $c_{\Jac C/\Q_{17}} = 1$ and $\mu_{C/\Q_{17}} =1$. 
Tamagawa numbers and local root numbers for both elliptic curves $E$ and $\Jac E'$ are $c_{E/\Q_{17}} = 1$, $c_{\Jac E'/\Q_{17}}=2$, $w_{E/\mathcal \Q_{17}}=1$ and $w_{\Jac E'/\mathcal \Q_{17}} =-1$. By Definition \ref{de:localterm} and Lemma \ref{lem:kercoker} it follows that $\lambda_{f,\Q_{17}} = -1$.  A simple computation shows that $H_{f,\Q_{17}} =+1$ which proves Theorem \ref{thm:localthm} for this particular choice of $E$.
\end{example}

\subsection{Proof of Theorem \ref{thm:localthm} when the cluster picture for $E'$ has either 0 or 1 twin only.}

\begin{proposition}\label{thm:localQp}Theorem \ref{thm:localthm} holds when $\mathcal K/\Q_p$ is a finite extension for some odd prime $p$, \\$f(x)\in\mathcal O_{\mathcal K}[x]$ and both $E$, $E'$ are semistable with $E'$ having cluster picture
either  {\smash{\raise2pt\hbox{\clusterpicture            
  \Root[] {1} {first} {r1};
  \Root[] {} {r1} {r2};
  \Root[] {} {r2} {r3};
    \Root[] {} {r3} {r4};
  \ClusterD c1[0] = (r1)(r2)(r3)(r4);
\endclusterpicture}}} or
 \clusterpicture            
  \Root[] {1} {first} {r2};
  \Root[] {} {r2} {r3};
  \Cluster c1 = (r2)(r3);
  \Root[] {1} {c1} {r4};
  \Root[] {} {r4} {r5};
  \ClusterD c3[0] = (r2)(c1)(r4)(r5);
\endclusterpicture.\end{proposition}

 \begin{proof} 

The input of Table $\ref{reductiontype}$ are the cluster pictures of $E$ and $E'$ (columns 2 and 3), coloured according to Notation \ref{not:drawcluster}, such that $E'$ has at most one double root. Column 1 indexes the various cases using the reduction types of $C/\K$ as defined in \cite{m2d2}, Table 1.1. 
  
     \begin{table}[h!]
\def\cltopskip{3pt}              
\def\clbottomskip{3pt}           
\def\graphdslabelscale{1.0}      
\def\cdepthscale{0.6}            
$$
\scalebox{0.78}{$
\begin{array}{| c|c|c||c|c|c|c|c|c|c|c|c|c|}
\hline
\text{Types} & \Sigma_{E/\mathcal K}  &\Sigma_{E'/\mathcal K}& \Sigma_{\Jac E'/\mathcal K}& \Upsilon_{C/\mathcal K}& c_{E/\mathcal K} & c_{\Jac E'/\mathcal K}& c_{\Jac C/\mathcal K} &\mu_{C/\mathcal K}& \lambda_{f,\mathcal K}& w_{E/\mathcal K}w_{\Jac E'/\mathcal K} & \E\\
\hline
2& \goodE  &\goodEprime & {\smash{\raise2pt\hbox{\Jgood}}}& \tgrGA & 1 & 1 & 1 & 1&+1&+1& +1\\
 \hline
 1_n^+ & \goodE  &\ctpEprime& \JCtwinP &\tgrGB & 1 & 2n & n & 1&-1& -1 &+1\\
 \hline
 1_n^- & \goodE  &\ctmEprime& \JCtwinM&\tgrGC& 1 & 2 & \tilde n &1 &(-1)^n&+1 & (-1)^n\\
 \hline
\text{I}_{n,n}^{+,+}  & \tpE  &\btpEprime&\JMtwinP& \tgrGD & n & n & n^2 &1 &+1&+1& +1\\
 \hline
\text{I}_{n\sim n}^{+}(a)  & \tpE  &\btmEprime&\JMtwinM& \tgrGG & n & \tilde n & n &1 &(-1)^{n+1}&-1  & (-1)^n\\
 \hline
\text{I}_{n\sim n}^{+}(b)  & \tmE  &\btpEprime&\JMtwinP& \tgrGG & \tilde n & n & n &1 &(-1)^{n+1}& -1 &(-1)^n\\
 \hline
\text{I}_{n,n}^{-,-}  & \tmE  &\btmEprime&\JMtwinM& \tgrGF & \tilde n & \tilde n & \tilde n^2 & 1&+1&+1  &+1\\
 \hline
\end{array}$}
$$
\hbox{\footnotesize{Notation: $\tilde x=2$ if $2|x$ and $\tilde x=1$ if $2\nmid x$. }}
\caption{Local data for $\mathcal K/\Q_p$ when $E'$ has at worst one double root}
\label{reductiontype}
\end{table}

Column 4 gives the cluster picture for $\Jac E'$, which is easily determined from that of $E'$ using Remark \ref{re:Jacobian}. 

Column 5 gives the dual graph of the minimal regular model of $C$, denoted $\Upsilon_{C/\mathcal K}$, where an arrow is used to indicate the action of Frobenius. This is determined using \cite[Theorems 3.1 \& 3.3]{omri} with $f_1(x) = x$, $f_2(x) = f(x)$ so that $Y$ becomes $C$.

Columns 6 and 7 list the Tamagawa numbers for $E$ and $\Jac E'$, calculated from their respective cluster pictures using \cite[Table 15.1]{silverman2009arithmetic}.

Similarly, column 8 contains the Tamagawa number for $\Jac C$ but this time calculated from $\Upsilon_{C/\mathcal K}$ using \cite[Lemma 2.22 \& Remark 2.23]{m2d2}.

Column 9 keeps track of the deficiency contribution from $C/\mathcal K$, using \cite[Theorem 12.4]{m2d2}. 

Column 10 gives the value of $\lambda_{f,\mathcal K}$ calculated using Lemma \ref{lem:kercoker}.

Column 11 gives $w_{E/\mathcal K}w_{\Jac E'/\mathcal K}$ using that for a semistable elliptic curve $E$, $w_{E/\mathcal K}=-1$ only when the curve has split multiplicative reduction, i.e. a twin with sign $+$.

It remains to compute the value of $\E$ via a case by case analysis of the valuations of $b$, $c$, $L$, $\Delta_f$.  Unless specified otherwise, equivalences are taken mod $\pi_{\mathcal K}$. By assumption $a, b, c\in\mathcal O_{\mathcal K}$. Recall that $\alpha_1, \alpha_2, \alpha_3$ denote the roots of $f$.

{\bf Type $2$.} We must show that $\E=1$. We have that $v(c)=v(\Delta_f)=0$. 

\textit{Suppose $b,L\neq0$.} Then $\E = (b,-c)_{\mathcal K}\cdot(L, -b\Delta_f)_{\mathcal K}$. If $v(b)=v(L)=0$ we are done. Suppose that $\mathcal K/\Q_p$ and $p\neq 3$. If $b\equiv 0$ then $L\equiv-9c$ so $\E = 1$. If $L\equiv 0$ then $v(a)=v(b)=0$ and $-b\Delta_f\equiv\frac{4}{b^2}(b^3-27c^2)^2$ so again $\E=1$. 

Now suppose that $\mathcal K/\Q_3$ and $L\equiv 0$. Write $b = \pi_{\mathcal K}^{v(b)}u_1$, $3=\pi_{\mathcal K}^{v(3)}u_2$ for units $u_1, u_2$ and assume that $v(b)>0$ (if $v(b)=0$ then we again observe that $-b\Delta_f$ is a square). Note that $\Delta_f\equiv-a^3c$ so $v(a)=0$.
	   \begin{itemize}
     \item If $v(b)<2v(3)$, then $L = \pi_{\mathcal K}^{v(b)}(au_1-\pi_{\mathcal K}^{2v(3)-v(b)}cu_2^2)$ where $au_1-\pi_{\mathcal K}^{2v(3)-v(b)}cu_2^2\equiv au_1$. So $\E = 1$ having used the standard Hilbert symbol identity $(\pi_{\mathcal K}, -\pi_{\mathcal K})_{\mathcal K} = 1$.
     \item If $v(b)=2v(3)$, then $L = \pi_{\mathcal K}^{v(b)}(au_1-cu_2^2)$ and $\E= (au_1-cu_2^2,-u_1\Delta_f)_{\mathcal K}$, having used that $v(b)$ is even. If $v(au_1-cu_2^2)=0$ then we are done, else $-u_1\Delta_f \equiv u_2^2a^2c^2$. So $\E = 1$.
       \item If $v(b)>2v(3)$, then $L = \pi_{\mathcal K}^{2v(3)}(\pi_{\mathcal K}^{v(b)-2v(3)}au_1-cu_2^2)$ where $\pi_{\mathcal K}^{v(b)-2v(3)}au_1-cu_2^2\equiv-cu_2^2$. So $\E=1$.
          \end{itemize}	

\textit{Suppose $bL=0$.} Clearly $\E=1$.

{\bf Types $1_n^+$ or $1_n^-$.} Suppose that $\alpha_1\equiv0$ with $v(\alpha_1) = n$. Since $\theta_{\mathfrak t, E'} = \sqrt{\frac{1}{4}(\alpha_1-2\alpha_2)(\alpha_1-2\alpha_3)}$, it can be observed that we are in type $1_n^+$ when $\alpha_2\alpha_3\in(\mathcal K^\times)^2$ (when $\theta_{\mathfrak t, E'}\in \mathcal K^\times$), and type $1_n^-$ otherwise.  We are required to show that $\E=-1$ precisely when $n$ is odd and $\alpha_2\alpha_3\notin(\mathcal K^\times)^2$, i.e. that $\E = (\pi_{\mathcal K}^n,  \alpha_2\alpha_3)_{\mathcal K}$. We have that $v(b) = v(\Delta_f) = 0$ and $v(c) = n$.

\textit{Suppose $L\neq0$.} Then $\E = (b,\pi_{\mathcal K}^n)_{\mathcal K}\cdot(L, -b\Delta_f)_{\mathcal K}$ and as $b\equiv \alpha_2\alpha_3$  it remains to show that $(L, -b\Delta_f)_{\mathcal K} = 1$. If $v(L)=0$ we are done. If not, then since $L\equiv -\alpha_2\alpha_2(\alpha_2+\alpha_3)$ we have that $\alpha_2 \equiv - \alpha_3$ and $-b\Delta_f\equiv4\alpha_3^8$.

\textit{Suppose $L=0$.} Then $\E = (\pi_K^n,-\Delta_f)_{\mathcal K}$ where $\Delta_f = -\frac{4}{b^3}(b^3-27c^2)^2$ and $b\equiv \alpha_2\alpha_3$.

{\bf Types $\mathrm{I}_{n,n}^{+,+}$, $\mathrm{I}_{n\sim n}^{+}(a)$, $\mathrm{I}_{n\sim n}^{+}(b)$, or $\mathrm{I}_{n,n}^{-,-}$.} Suppose that $v(\alpha_1)=0$ and $\alpha_2\equiv\alpha_3$ with $v(\alpha_2-\alpha_3) = \frac{n}{2}.$ 
Since $\theta_{\mathfrak t, E} = \sqrt{\frac{1}{2}(\alpha_2+\alpha_3)-\alpha_1}$ and $\theta_{\mathfrak t, E'} = \sqrt{\frac{1}{2}(\alpha_2+\alpha_3)(\frac{1}{2}(\alpha_2+\alpha_3)-\alpha_1)}$, it can be observed that we are in type $\mathrm{I}_{n,n}^{+,+}$ or $\mathrm{I}_{n,n}^{-,-}$ when $\frac{1}{2}(\alpha_2+\alpha_3)\in (\mathcal K^\times)^2$ (when $\theta_{\mathfrak t, E},\theta_{\mathfrak t, E'}\in \mathcal K^\times$, or $\theta_{\mathfrak t, E},\theta_{\mathfrak t, E'}\not\in \mathcal K^\times$), and type $\mathrm{I}_{n,n}^{+,-}$ or $\mathrm{I}_{n,n}^{-,+}$ otherwise. We are required to show that $\E = -1$ precisely when $n$ is odd and $\frac{1}{2}(\alpha_2+\alpha_3)\notin (\mathcal K^\times)^2$, i.e. that $\E = (\pi_{\mathcal K}^n,  \frac{1}{2}(\alpha_2+\alpha_3))_{\mathcal K}$. We have that $v(c) = v(L) = 0$ and $v(\Delta_f) = n$.

\textit{Suppose $b\neq0$.} Then $\E = (b, -cL)_{\mathcal K}\cdot(-2L, \pi_{\mathcal K}^n)_{\mathcal K}$ and as $L\equiv -(\alpha_2+\alpha_3)(\alpha_1-\frac{1}{2}(\alpha_2+\alpha_3))^2$ it remains to show that $(b,-cL) = 1$. If $v(b)=0$ we are done. If not,
 since $b\!\equiv\! \frac{1}{2}(\alpha_2+\alpha_3)(2\alpha_1+\frac{1}{2}(\alpha_2+\alpha_3))$ we have that $\alpha_1\equiv-\frac{1}{4}(\alpha_2+\alpha_3)$ and $-cL\equiv\frac{9}{256}(\alpha_2+\alpha_3)^6$.

\textit{Suppose $b=0$.} Then $\E = (2c,\pi_{\mathcal K}^n)_{\mathcal K}$ where $2c = 2\alpha_1^2(\alpha_2+\alpha_3)$. \end{proof}
\section{Local Theorem in all remaining cases}\label{s:thmproof}

Here we prove that Theorem \ref{thm:localthm} holds in all remaining cases using a global--local argument. In summary, so far we have proved the following:

\begin{proposition}Let $\mathcal K$ be a local field of characteristic 0 and let $f(x)\in \mathcal K[x]$ be a separable monic cubic polynomial with $f(0)\neq0$. 
Theorem \ref{thm:localthm} holds when \begin{enumerate} \item $\mathcal K = \C$,
\item $\mathcal K = \R$, or
\item $\mathcal K/\Q_p$ is a finite extension for an odd prime $p$, $f(x)\in \mathcal O_{\mathcal K}[x]$, $E, E'/\mathcal K$ are both semistable and $xf(x)$ has at worst one double root $\text{mod } \pi_{\mathcal K}$.\end{enumerate}  
\end{proposition}

\begin{proof}These are Propositions \ref{thm:localcomplex}, \ref{thm:localreals} and \ref{thm:localQp} respectively.\end{proof}

\subsection{Continuity of local invariants}\label{s:continuity}

 \begin{lemma}[\cite{dokchitser2011root} Lemma 3.2]\label{lem:HScontinuity} Let $\mathcal K$ be a local field of characteristic 0. The Hilbert symbol $(A, B)_{\mathcal K}$ is a continuous function of $A, B\in \mathcal K^\times$. \end{lemma}
 
        \begin{lemma}\label{lem:invariantcontinuity} Let $\mathcal K$ be a local field of characteristic 0 and $f(x) = x^3+ax^2+bx+c\in \mathcal K[x]$ a separable cubic. The invariants $b$, $c$, $L:=ab-9c$, $\Delta_f$, $w_{E/\mathcal K}$, $w_{\Jac E'/\mathcal K}$ and $\lambda_{f,\mathcal K}$ are continuous in the coefficients of $f(x)$. \end{lemma}
        
        \begin{proof}This can be seen from \cite[Lemma 9.2]{parityforab}. \end{proof}
        
        \begin{lemma}\label{lem:thmcontinuity} Let $\mathcal K$ be a local field of characteristic 0 and $f(x)\in \mathcal O_{\mathcal K}[x]$ be a separable monic cubic. If $\tilde f(x)\in\mathcal O_{\mathcal K}[x]$ is a separable monic cubic with coefficients close enough to those of $f$ then Theorem \ref{thm:localthm} holds for $f$ if and only if it holds for $\tilde f.$ \end{lemma}
        
        \begin{proof} Write $f(x) = x^3 + ax^2+bx+c$, $\tilde f(x) = x^3 + \tilde ax^2+\tilde bx+\tilde c$. If $b,L\neq 0$ then, ensuring that $\tilde b, \tilde L\neq 0$, this is clear from Lemmas \ref{lem:HScontinuity} and \ref{lem:invariantcontinuity}. If $b = \tilde b = 0$ or $L = \tilde L = 0$ then again this is clear. When we are not in either of these cases, Lemma \ref{lem:invariantcontinuity} still asserts that the root numbers are unchanged and that $\lambda_{f/\mathcal K} = \lambda_{\tilde f/\mathcal K} $, but showing that $\E = H_{\tilde f,\mathcal K}$ is more delicate.
        
        Suppose that $\mathcal K$ is non-archimedean. Let $\pi_{\K}$ be a choice of uniformiser and $v$ a normalised valuation. We write $\square$ for a non-zero square element in $\K$. 
        
        Suppose that $b = 0$ and $\tilde b \neq 0$. Let $N = v(c)+v(36)+1$ and pick $a\equiv \tilde a, b\equiv \tilde b, c\equiv \tilde c \mod \pi_{\mathcal K}^N$. Then $\tilde L = \tilde a \tilde b - 9\tilde c\equiv -9\tilde c \not\equiv 0 \mod \pi_{\mathcal K}^N$ and so $\tilde L = -\tilde c\cdot\square.$ Therefore 
        $H_{\tilde f,\mathcal K} =(\tilde b,-\tilde c)_{\mathcal K}\cdot(2\tilde c, \Delta_{\tilde f})_{\mathcal K}\cdot(-\tilde c, -\tilde b)_{\mathcal K} = (2\tilde c, \Delta_{\tilde f})_{\mathcal K}\cdot(-\tilde c, -1 )_{\mathcal K}$ and by Lemma \ref{lem:HScontinuity} this is equal to $\E$.
        
        Now suppose that $L = 0$ and $\tilde L \neq 0$. Let $N = v(b)+2v(a^2-3b)+v(16)+1$ and pick $a \equiv \tilde a$, $b\equiv \tilde b, c\equiv \tilde c \mod \pi_{\mathcal K}^N$. We have that $\tilde a\tilde b\equiv 9\tilde c\mod \pi_{\mathcal K}^N$, therefore $9\Delta_{\tilde f} \equiv -4\tilde b(\tilde a^2-3\tilde b)^2\not\equiv 0\mod\pi_{\mathcal K}^N$ and $\Delta_{\tilde f} = -\tilde b \cdot\square$. So, $H_{\tilde f,\mathcal K} =(\tilde b,-\tilde c)_{\mathcal K}\cdot(-2\tilde L, -\tilde b)_{\mathcal K}\cdot(\tilde L, -\tilde b)_{\mathcal K} = (\tilde b,-\tilde c)_{\mathcal K}\cdot(-2, -\tilde b)_{\mathcal K}$ and by Lemma~\ref{lem:HScontinuity} this is equal to $\E$ since $\Delta_f = -b\cdot\square$.
        
        The case when $\mathcal K$ is archimedean and $bL=0$ follows from Table \ref{table:reals} since $f$ and $\tilde f$ will have the same number of positive and negative real roots. \end{proof}
        
        \subsection{Reduction steps}\label{s:reduction}
        
        \begin{lemma}\label{lem:integralcoeffs}Let $\mathcal K$ be a non-archimedean local field of characteristic 0. It is sufficient to prove Theorem \ref{thm:localthm} for $f(x)\in \mathcal O_{\mathcal K}[x]$. \end{lemma}
        
        \begin{proof} Let $g(x) = x^3+ax^2+bx+c\in \mathcal K[x]$ and choose $d\in\mathcal K^\times$ such that $d^2a, d^4b, d^6c\in\mathcal O_{\mathcal K}$. Define $f(x) = x^3+d^2ax^2+d^4bx+d^6c\in\mathcal O_{\mathcal K}[x]$. Since $y^2 = g(x) \cong y^2 = f(x)$, $y^2 = xg(x) \cong y^2 = xf(x)$ and $H_{g,\mathcal K} = \E$ (the Hilbert symbol entries have been scaled by squares), Theorem \ref{thm:localthm} holds for $g(x)$ if and only if it holds for $f(x)$.\end{proof}

            \begin{lemma}\label{lem:clusterreduction}Let $K$ be a number field and $f(x) = x^3+ax^2+bx+c\in\mathcal{O}_K[x]$ a separable cubic such that $a^2-3b, b, a^2-4b, ab-9c, c\neq0$. Fix a prime $\mathfrak p\nmid 2,3$ and suppose that the following conditions are satisfied:
         \begin{enumerate} \item if $\mathfrak p\mid a^2-3b$ then $9c\not\equiv ab\mod \mathfrak p$
         \item if $\mathfrak p\mid b(a^2-4b)$ then $c\not\equiv 0\mod \mathfrak p$
         \end{enumerate}
         then $xf(x)$ has either distinct roots or no worse than 1 double root modulo $\mathfrak p$.\end{lemma}
         
         \begin{proof} We must prove that (1) and (2) guarantee that $xf(x)$ has at least 3 distinct roots mod $\mathfrak p$.
         
         Suppose that three roots of $xf(x)$ are congruent mod $\mathfrak p$. This happens when either (i) all three roots of $f(x)$ are congruent, or (ii) two roots of $f(x)$ are congruent to 0. Situation (i) occurs if $f(x)$, $f'(x) = 3x^2+2ax+b$, $f''(x)=6x+2a$ share a root mod $\mathfrak p$, i.e. when $f(-\frac{a}{3})\equiv f'(-\frac{a}{3})\equiv0\mod \mathfrak p$. Condition (1) ensures that this does not happen. We are in situation (ii) whenever $f(0)\equiv f'(0)\equiv 0 \mod \mathfrak p$. Condition (2) ensures that this does not happen. 
         
         Now suppose that $xf(x)$ has two double roots mod $\mathfrak p$, so that two roots of $f(x)$ are (nonzero and) congruent mod $\mathfrak p$ with the remaining root congruent to zero. Here $f(0)\equiv 0$ and $\Delta_f\equiv 0$. By definition of $\Delta_f$, these happen simultaneously when $c\equiv 0$ and $a^2b^2 -4b^3\equiv 0$, so again condition (2) ensures that this does not happen. \end{proof}
         
         \begin{remark}Using cluster pictures (see \S \ref{ss:clusterpics}), Lemma \ref{lem:clusterreduction} says that $\Sigma_{E'/K_{\mathfrak p}}$ is either {\smash{\raise2pt\hbox{\clusterpicture            
  \Root[] {1} {first} {r1};
  \Root[] {} {r1} {r2};
  \Root[] {} {r2} {r3};
    \Root[] {} {r3} {r4};
  \ClusterD c1[0] = (r1)(r2)(r3)(r4);
\endclusterpicture}}} or  \clusterpicture            
  \Root[] {1} {first} {r2};
  \Root[] {} {r2} {r3};
  \Cluster c1 = (r2)(r3);
  \Root[] {1} {c1} {r4};
  \Root[] {} {r4} {r5};
  \ClusterD c3[0] = (r2)(c1)(r4)(r5);
\endclusterpicture. In particular, Proposition \ref{thm:localQp} holds for such a choice of $f$. \end{remark}

\begin{lemma}Let $\mathcal K/\Q_3$ be a finite extension and let $f(x) = x^3+x+1\in \mathcal K[x]$ then Theorem \ref{thm:localthm} holds.\end{lemma}

\begin{proof} $E$ and $\Jac E'$ have good reduction at primes $\mathfrak p|3$ so $w_{E/\mathcal K} = w_{\Jac E'/\mathcal K} = 1$ and 
$\lambda_{f,\mathcal K} = 1.$ That $\E = 1$ follows by definition since $\E = (1,-1)_{\mathcal K}\cdot(18, -31)_{\mathcal K}\cdot(-9,-1)_{\mathcal K}$.
\end{proof}

\begin{proof}[Proof of Theorem \ref{thm:localthm}] We first prove that the theorem holds for $\mathcal K/\Q_2$. Pick a totally real number field $F$ with a unique prime $\mathfrak p\mid2$ so that $F_{\mathfrak p} \cong \mathcal K$ (to see that such a field exists, if $\mathcal K = \Q_2(\theta)$ and $h(x)\in\Q_2[x]$ is the minimal polynomial for $\theta$ then approximate $h(x)$ by some $\tilde h(x) \in \Q[x]$ which has all real roots; then take $F = \Q[x]/\tilde h(x)$). Fix another prime $\mathfrak p\neq \mathfrak p'\nmid 2,3$. By Lemma \ref{lem:integralcoeffs}, we may assume that $f(x) = x^3+ax^2+bx+c \in \mathcal O_{\mathcal K}[x]$. We will approximate $f(x)$ by $\tilde f(x) = x^3 + \tilde ax^2 + \tilde bx + \tilde c\in\mathcal O_F[x]$ subject to the following:
\begin{enumerate}[(i)] \item pick $\tilde a$ to be $\mathfrak p$-adically close to $a$, $\mathfrak p'$-adically close to $-1$ and $\mathfrak q$-adically close to 0 for all $\mathfrak q|3$,
\item pick $\tilde b \neq 0, \frac{1}{4}\tilde a^2, \frac{1}{3}\tilde a^2$ so that $\tilde b$ is $\mathfrak p$-adically close to $b$, $\mathfrak p'$-adically close to $-1$ and $\mathfrak q$-adically close to $1$ for all $\mathfrak q\mid3$,
\item pick $\tilde c \neq 0, \frac{1}{9}\tilde a\tilde b$ so that $\tilde c$ is $\mathfrak p$-adically close to $c$, $\mathfrak p'$-adically close to $1$, $\mathfrak q$-adically close to $1$ for all $\mathfrak q\mid3$, and such that if $\mathfrak q\nmid 2,3, \mathfrak p'$ then conditions (1) and (2) of Lemma \ref{lem:clusterreduction} are satisfied (namely: $9\tilde c\not\equiv \tilde a\tilde b\mod \mathfrak q$ for all $\mathfrak q\mid \tilde a^2-3\tilde b$, and $\tilde c\not\equiv 0\mod \mathfrak q$ for all $\mathfrak q\mid \tilde b(\tilde a^2-4\tilde b)$). 
\end{enumerate} 

Let $\tilde E: y^2 = \tilde f(x)$ and $\tilde E': y^2=x\tilde f(x)$. By construction, $\tilde f(x)$ is $\mathfrak p$-adically close to $f(x)$, $\mathfrak p'$-adically close to $x^3-x^2-x+1$, $\mathfrak q$-adically close to $x^3+x+1$ when $\mathfrak q|3$, and $\mathfrak q$-adically close to a monic cubic for which Proposition \ref{thm:localQp} holds when $\mathfrak q\nmid 2,3,\mathfrak p'$. Invoking Lemma \ref{lem:thmcontinuity}, this means that Theorem \ref{thm:localthm} holds for $\tilde f\in F_v[x]$ whenever $v\neq\mathfrak p$, and proving it for $\tilde f\in F_{\mathfrak p}[x]$ is equivalent to proving it for $f\in F_{\mathfrak p}[x].$

Since $\ord_{\mathfrak p'}(j(\tilde E)), \ord_{\mathfrak p'}(j(\tilde E'))<0$, and the $2$-parity conjecture holds for elliptic curves over totally real number fields with non-integral $j$-invariant \cite[Theorem 2.4]{dokchitser2011root}, it holds for $\tilde E$ and $\tilde E'$. Therefore \begin{alignat*}{2}1 &= (-1)^{\rank_2\tilde E/F+\rank_2\Jac\tilde E'/F}w_{\tilde E/F}w_{\Jac\tilde E'/F}&\\
&=\prod_{v}\lambda_{\tilde f,F_v}w_{\tilde E/F_v}w_{\Jac\tilde E'/F_v}&(\text{= Theorem \ref{thm:paritythm}})\\
&=\lambda_{\tilde f,F_{\mathfrak p}}w_{\tilde E/F_{\mathfrak p}}w_{\Jac\tilde E'/F_{\mathfrak p}}\prod_{\mathfrak p\neq v}H_{\tilde f,F_v}&(\text{= Theorem \ref{thm:localthm}})\\
&=\lambda_{\tilde f,F_{\mathfrak p}}w_{\tilde E/F_{\mathfrak p}}w_{\Jac\tilde E'/F_{\mathfrak p}}H_{\tilde f,F_{\mathfrak p}}&(\text{= Hilbert symbol product law})
\end{alignat*} 
where the products range over all places of $F$. In conclusion, we now know that Theorem \ref{thm:localthm} holds for $\tilde f\in F_{\mathfrak p}[x]$ and so it must also hold for $f\in F_{\mathfrak p}[x]$ where $F_{\mathfrak p}\cong \mathcal K$. 

Using that Theorem \ref{thm:localthm} now holds for $\mathcal K/\Q_2$, we can prove it for $\mathcal K/\Q_p$ ($p$ odd) by repeating the above argument after replacing 2 by $p$ (when $p=3$ we also replace the condition ``$\mathfrak q\mid3$'' by ``$\mathfrak q\mid3$ and $\mathfrak q\neq \mathfrak p$''). \end{proof}

\section{Global results}\label{s:global}

Here we present the 2-parity results we obtain as consequences to Theorem \ref{thm:localthm}.

\begin{theorem}\label{thm:2PCforEC} Let $K$ be a number field and let $f(x)\in K[x]$ be a separable monic cubic polynomial with $f(0)\neq0$. 
The $2$-parity conjecture holds for $E/K$ if and only if it holds for $\Jac E'/K$.\end{theorem}

\begin{proof} Theorem \ref{thm:localthm} says that for each place $v$ of $K$, $w_{E/K_v}w_{\Jac E'/K_v} = \lambda_{f,K_v} H_{f,K_v}$. Taking the product over all such $v$ then invoking Theorem \ref{thm:paritythm} and the product formula for Hilbert symbols, we see that \begin{equation*}w_{E/K}w_{\Jac E'/K} = (-1)^{\rank_2 E/K+\rank_2 \Jac E'/K}.\qedhere\end{equation*}
\end{proof}

\begin{theorem}\label{thm:2PCsplit}Let $K$ be a number field. The $2$-parity conjecture holds for $\Jac C/K$ whenever $C$ is a genus 2 hyperelliptic curve of the form $C: y^2 = cf(x^2)$ with $c\in K^\times$ and $f(x)\in K[x]$ a separable monic cubic polynomial with $f(0)\ne0$.\end{theorem}

\begin{proof} Let $E_0:y^2 = cf(x)$, $E_0':y^2 = cxf(x)$. Adapting the proof of Lemma \ref{lem:isogeny} to incorporate the leading term $c\in K^\times$ in $C: y^2 = cf(x^2)$, shows that there is an isogeny $E_0\times \Jac E_0'\rightarrow \Jac C $. It follows that proving the $2$-parity conjecture for $\Jac C/K$ is equivalent to proving it for $E_0\times \Jac E_0'/K$. Define $E: y^2 = f(x)$ and $E':y^2 = xf(x)$. Observe that the $2$-parity conjecture for quadratic twists \cite[Corollary 1.6]{dokchitser2011root} says that the $2$-parity conjecture holds for $E_0/K$ (resp. $\Jac E_0'/K$) if and only if it holds for $E/K$ (resp. $\Jac E'/K$). This reduces proving the $2$-parity conjecture for $\Jac C/K$ to proving it for $E\times \Jac E'/K$, which is precisely Theorem \ref{thm:2PCforEC}.\end{proof}

\begin{lemma}\label{lem:2tors} Let $K$ be a number field and $E_1, E_2/K$ be elliptic curves. If $E_1[2]\cong E_2[2]$ as Galois modules, then either $E_2$ is a quadratic twist of $E_1$ or there exists separable monic $f(x)\in K[x]$ with $f(0)\neq0$ and $d\in K$ such that we have the following isomorphisms of curves
$$E_1\cong y^2 = f(x),\qquad\qquad E_2\cong dy^2=xf(x).$$\end{lemma}

\begin{proof} Write $E_1:y^2=g_1(x)$, $E_2:y^2 = g_2(x)$ for monic cubics $g_1, g_2\in K[x]$ and $\psi: E_1[2]\to E_2[2]$ for a Galois module isomorphism. Let $\alpha_1, \alpha_2, \alpha_3$ and $\beta_1, \beta_2, \beta_3$ be the roots of $g_1$ and $g_2$ respectively, labelled so that $\psi((\alpha_i,0)) = (\beta_i,0)$. Let 
\begin{align*} A&= \alpha_1\alpha_2(\beta_1-\beta_2)+\alpha_3\alpha_1(\beta_3-\beta_1) + \alpha_2\alpha_3(\beta_2-\beta_3),\\
B&=\alpha_1\alpha_2\beta_3(\beta_2-\beta_1)+\alpha_3\alpha_1\beta_2(\beta_1-\beta_3) + \alpha_2\alpha_3\beta_1(\beta_3-\beta_2),\\
C&=\beta_1(\alpha_2-\alpha_3)+\beta_2(\alpha_3-\alpha_1)+\beta_3(\alpha_1-\alpha_2),\\
D&=\beta_1\beta_2(\alpha_1-\alpha_2)+\beta_2\beta_3(\alpha_2-\alpha_3)+\beta_3\beta_1(\alpha_3-\alpha_1),
\end{align*} so that $h(z):= \frac{Dz-B}{A-Cz}$ is the M\"obius transformation mapping $\alpha_i$ to $\beta_i$. Using that Galois permutes the roots of $g_1$ and the roots of $g_2$ in the same way, one can readily check that $h(z)$ is defined over $K$. Observing that $h(x) - h(\alpha_i)  = \frac{AD-BC}{(A-Cx)(A-C\alpha_i)}(x-\alpha_i)$ gives
$$g_2(h(x)) = \frac{(AD-BC)^3}{(A-Cx)^3(A-C\alpha_1)(A-C\alpha_2)(A-C\alpha_3)}g_1(x).$$ Clearly, $E_2\cong y^2 = g_2(h(x))$ and hence also,
$$E_2\cong y^2 = \begin{cases}\frac{A(A-C\alpha_1)(A-C\alpha_2)(A-C\alpha_3)}{AD-BC}(1-\frac{C}{A}x)g_1(x)&A\neq0,
\\-\frac{C\alpha_1\alpha_2\alpha_3}{B}xg_1(x) &A=0.
\end{cases}$$ 

If $A,C\neq 0$, set $x' = 1-\frac{C}{A}x$ then $E_1\cong y^2 = f(x')$ and $E_2\!\cong\!y^2\!=\!dx'f(x')$ where $f(x')\!=\!g_1(\frac{A}{C}(1-x'))$ and $d\in K^\times$. If $A\neq 0$ and $C=0$, then $E_2$ is just a quadratic twist of $E_1$. Finally, if $A=0$ then the result is clear. \end{proof}

\begin{theorem}\label{thm:2tors} Let $K$ be a number field and $E_1, E_2/K$ be elliptic curves. If $E_1[2]\cong E_2[2]$ as Galois modules, then the $2$-parity conjecture holds for $E_1/K$ if and only if it holds for $E_2/K$.\end{theorem}

\begin{proof} By Lemma \ref{lem:2tors}, either $E_2$ is a quadratic twist of $E_1$ and this is just \cite[Corollary 1.6]{dokchitser2011root}, or $E_1\cong y^2 = f(x)$ and $E_2\cong dy^2=xf(x)$ (as curves) for some separable monic $f(x)\in K[x]$ with $f(0)\neq 0$ and $d\in K$. Since $E_2\cong \Jac E_2 \cong \mathrm{Jac}(dy^2 = xf(x))$, \cite[Corollary 1.6]{dokchitser2011root} tells us that the 2-parity conjecture holds for $E_2/K$ if and only if it holds for $ \mathrm{Jac}(y^2 = xf(x))/K$. Theorem \ref{thm:2PCforEC} says that the 2-parity conjecture holds for $\mathrm{Jac}(y^2 = xf(x))/K$ if and only if it holds for $E_1/K$. \end{proof}


When working with elliptic curves over totally real number fields, the $p$-parity conjecture is known in most cases thanks to the work of Nekov\'a\v r (\cite[Theorem E]{Nekovar3},\cite[Theorem A]{Nekovar1}) and Dokchitser--Dokchitser (\cite[Theorem 2.4]{dokchitser2011root}). When $p=2$ and $E$ does have complex multiplication some cases are proved in (\cite[(5.9)]{Nekovar2}) but this setting is unsolved in general. Theorem \ref{thm:2tors} allows us to prove the missing cases, completing the proof of the $p$-parity conjecture for elliptic curves over totally real fields.

\begin{theorem}\label{2PCtotallyreal} Let $E/K$ be an elliptic curve over a totally real number field with complex multiplication. The $2$-parity conjecture holds for $E/K$. \end{theorem}

\begin{proof}Suppose $E:y^2 = f(x)$ where $f(x) = (x-\alpha_1)(x-\alpha_2)(x-\alpha_3)$. For $\gamma \in K$, set $f_\gamma(x) =  f(x-\gamma)$ and define $E_\gamma = \Jac(y^2 = xf_\gamma(x))$. By remark \ref{re:Jacobian}, a model for $E_\gamma$ is given by
$$y^2 = (x+(\alpha_1+\gamma)(\alpha_2+\gamma))(x+(\alpha_1+\gamma)(\alpha_3+\gamma))(x+(\alpha_2+\gamma)(\alpha_3+\gamma)).$$
Clearly $E[2]\cong E_\gamma[2]$ as Galois modules. The $j$-invariant for $E_\gamma$ is a non-constant rational function in $\gamma$, so for some $\gamma_0\in K$ it will hold that $j(E_{\gamma_0})\not\in\mathcal O_K$, therefore $E_{\gamma_0}$ does not have CM. In particular, the $2$-parity conjecture holds for $E_{\gamma_0}/K$ by \cite[Theorem 2.4]{dokchitser2011root}, and applying Theorem \ref{thm:2tors} gives that the $2$-parity conjecture holds for $E/K$.
\end{proof} 

\begin{corollary}\label{maincorollary}Let $p$ be a prime and $K$ be a totally real field. The $p$-parity conjecture holds for all elliptic curves over $K$.\end{corollary}

\begin{proof} When $p$ is odd, this is \cite[Theorem E]{Nekovar3} and \cite[Theorem A]{Nekovar1}. When $p=2$ and the elliptic curve does not have complex multiplication, this is \cite[Theorem 2.4]{dokchitser2011root}. When $p=2$ and the elliptic curve has complex multiplication, this is Theorem \ref{2PCtotallyreal}.\end{proof}

\appendix

\section{Local Conjecture for a family of hyperelliptic curves}\label{gener}
\begin{center}
 by HOLLY GREEN
\end{center}
\medskip
As mentioned in \S1, finding the correct local discrepancy using the above approach is very challenging due to the current lack of a conceptual understanding for it. In fact, there are instances where one can give local formulae for the parity of the $2^{\infty}$-Selmer rank of varieties but where we do not know an expression for the local error term (see \cite{Jordan} for a local formula in the case of Jacobians of genus 2 or 3 curves with a $K$-rational 2-torsion point, and  Remark 1.8 for a discussion on the error term).

The error term found in this article stands out as it can be generalised in the settings presented below and it recovers the error term found in \cite{DDparity} when proving the $2$-parity conjecture for elliptic curves admitting a 2-isogeny. Here we motivate and present the generalisation, providing references for additional details on some of the results.

\subsection{Setup}

Let $K$ be a number field. For a separable monic polynomial $f(x)\in K[x]$ such that $f(0)\neq 0$ we consider the hyperelliptic curves 
 $$C_1 :y^2 = f(x),\qquad\qquad C_2: y^2 = xf(x),\qquad\qquad C:y^2=f(x^2).$$
 
As in \S 2, we can control the parity of the $2^\infty$-Selmer rank of $\Jac C_1 \times \Jac C_2$ using an isogeny.

\begin{lemma}\label{lem:isogeny2} Let $K$ be a field of characteristic $0$ and let $f(x)\in K[x]$ be separable, monic and such that $f(0)\neq0$. Then there is an isogeny $\phi : \Jac C_1 \times \Jac C_2\rightarrow \Jac C$, such that $\phi\phi^t = [2]$. \end{lemma}

\begin{proof} This is similar to the proof of lemma \ref{lem:isogeny}. In particular, it can be shown that the maps $C \to C_1: (x,y) \mapsto (x^2,y)$ and $C \to C_2:  (x,y) \mapsto (x^2,xy)$ induce an isogeny $\Jac C_1 \times \Jac C_2 \to \Jac C$ with the required properties.
\end{proof}

\begin{definition}\label{de:localterm2}Let $\mathcal K$ be a local field and let $f(x)\in \mathcal K[x]$ be separable, monic and such that $f(0)\neq0$. We write $$
\lambda_{f,\mathcal K} = \mu_{C_1/\mathcal K}\mu_{C_2/\mathcal K}\mu_{C/\mathcal K}\cdot (-1)^{\dim_{\F_2}\ker\phi|_{\mathcal K} - \dim_{\F_2}\coker\phi|_{\mathcal K}},
$$
where for a curve $X/\mathcal K$, $\mu_{X/\mathcal K}$ is $-1$ if $X$ is deficient over $\mathcal K$ (as defined in \cite[\S 8]{Poonen}) and $1$ otherwise.\\
Here $\ker\phi|_{\mathcal K} = (\Jac C_1 \times \Jac C_2)(\mathcal K)[\phi]$ and $\coker\phi|_{\mathcal K} = |\Jac C(\mathcal K)/\phi((\Jac C_1 \times \Jac C_2)(\mathcal K))|$.\end{definition}

\begin{theorem}\label{thm:paritythm2}
Let $K$ be a number field and let $f(x)\in K[x]$ be separable, monic and such that $f(0)\neq0$. Then
$$
  (-1)^{\rk_2 \Jac C_1/K+\rk_2 \Jac C_2/K}= \prod_v \lambda_{f,K_v},
$$
where the product runs over all places of $K$.
\end{theorem}

\begin{proof}This is argued as in \cite[Theorem 3.2]{parityforab} with $A =\Jac C_1 \times \Jac C_2$ and $A' = \Jac C$. \end{proof}

\subsection{Sturm polynomials}
As in Lemma \ref{lem:kercoker}, computing the quantity $\lambda_{f,\R}$ boils down to counting connected components of $\Jac C_1, \Jac C_2$ and $\Jac C$ over the reals. These numbers can be determined from  
the number of real roots of $f(x)$ and $f(x^2)$. To count the latter we use Sturm's theorem to determine the number of real roots of $f(x)$ greater than 0. 

\begin{definition} \label{de: sturmdef} Let $f(x)\in \R[x]$. The {\it Sturm sequence} for $f$ is a sequence of polynomials $P_0, P_1, \ldots $ defined via \begin{equation*}P_{0}=f(x),\qquad\qquad P_{1}=f'(x),\qquad\qquad P_{i+1}\equiv-P_{i-1}\textup{ (mod }P_i)\text{ for }i\geq 1,\end{equation*} where either $\deg P_{i+1}<\deg P_i$ or $P_{i+1} = 0$. The sequence terminates once one of the $P_i$ is zero.\end{definition}

\begin{definition} Let $f(x)\in \R[x]$ have Sturm sequence $P_0, P_1, \ldots $. We write $\sigma(x_0)$ for the number of sign changes (ignoring zeros) in the sequence of real numbers $$P_0(x_0), P_1(x_0), \ldots .$$\end{definition}

\begin{theorem}[Sturm's theorem] \label{Sturm} Let $f(x)\in\R[x]$ have Sturm sequence $P_0, P_1, \ldots $. The number of roots of $f$ in the interval $(a,b]\subset \R$ is $\sigma(a)-\sigma(b)$.
\end{theorem}
In practice, the Sturm sequence for a polynomial provides us with a way to compute the number of roots it has in any half-open interval. Therefore, the Sturm polynomials evaluated at 0 and $\infty$ were good candidates for the entries of the Hilbert symbol of our error term. 

\subsection{Conjecture based on experimental data}

\begin{definition}\label{de:errorterm2} Let $\mathcal K$ be a local field of characteristic 0 and let $f(x) \in \mathcal K[x]$ be separable and monic. Let $P_0, P_1, \ldots $ denote the Sturm sequence for $f$, write $c_i$ for the lead coefficient of $P_i$ and assume that $\deg P_i = \deg f-i$ for $i = 0,\ldots,\deg f$. Whenever $\prod _{i=0}^{\deg f-1}P_i(0) \neq0$, we define $H_{f,\mathcal K}$ to be the following product of Hilbert symbols
$$H_{f,\mathcal K} = \prod_{i=0}^{\deg f-2}(-P_i(0),P_{i+1}(0))_{\mathcal K}\cdot(c_i,-c_{i+1})_{\mathcal K}.$$\end{definition}

\begin{remark}When $f$ is a quadratic, this recovers the error term given in \cite{DDparity} (the isogeny $\phi$ given in lemma \ref{lem:isogeny2} becomes a 2-isogeny of elliptic curves) and when $f$ is a cubic, this is precisely $\E$ as defined in \ref{de:errorterm} (when $b,L\neq0$). \end{remark}

\begin{conjecture}[Local conjecture]\label{conj:localthm2} Let $\mathcal K$ be a local field of characteristic 0 and let $f(x)\in \mathcal K[x]$ be separable monic and such that $f(0)\neq 0$. Define hyperelliptic curves over $\mathcal K$ by $$C_1:y^2 = f(x),\qquad\qquad\qquad C_2:y^2 = xf(x).$$ Let $P_0, P_1, \ldots $ denote the Sturm sequence for $f$. If $\deg P_i = \deg f-i$ for $i = 0,\ldots,\deg f$ and $\prod _{i=0}^{\deg f-1}P_i(0) \neq0$, then
$$w_{\Jac C_1/\mathcal K}w_{\Jac C_2/\mathcal K} = \lambda_{f,\mathcal K} H_{f,\mathcal K}.$$ 
\end{conjecture}

\begin{remark}If $P_i(0) = 0$ or $\deg P_{i+1}< \deg P_i - 1$ for some $i$, then $H_{f,\mathcal K}$ is no longer well-defined. It is expected that a different Hilbert symbol expression can be found so that the conjecture still holds, as in definition \ref{de:errorterm}. \end{remark}

\begin{remark} We have proved the conjecture whenever $\mathcal K\cong \C$ or $\mathcal K\cong \R$. Proving it when $\mathcal K$ is a non-archimedean field is work in progress. 
When $f$ is a quartic and $\mathcal K/ \Q_p$ is a finite extension for odd $p$, we have an analogue of proposition \ref{thm:localQp}, namely the conjecture is known to hold when the cluster picture for $C_2$ is {\smash{\raise2pt\hbox{\clusterpicture            
  \Root[] {1} {first} {r1};
  \Root[] {} {r1} {r2};
  \Root[] {} {r2} {r3};
    \Root[] {} {r3} {r4};
        \Root[] {} {r4} {r5};
  \ClusterD c1[0] = (r1)(r2)(r3)(r4)(r5);
\endclusterpicture}}} or {\clusterpicture            
  \Root[] {1} {first} {r2};
  \Root[] {} {r2} {r3};
  \Cluster c1 = (r2)(r3);
  \Root[] {1} {c1} {r4};
  \Root[] {} {r4} {r5};
    \Root[] {} {r5} {r6};
  \ClusterD c3[0] = (r2)(c1)(r4)(r5)(r6);
\endclusterpicture} or
\clusterpicture            
  \Root[] {1} {first} {r2};
  \Root[] {} {r2} {r3};
  \Cluster c1 = (r2)(r3);
  \Root[] {1} {c1} {r4};
  \Root[] {} {r4} {r5};
        \Cluster c2 = (r4)(r5);
    \Root[] {1} {c2} {r6};
  \ClusterD c3[0] = (r2)(c1)(c2)(r6);
\endclusterpicture. \end{remark}

\begin{remark}\label{re:2PCforhyp} Supposing that conjecture \ref{conj:localthm2} is true, the 2-parity conjecture holds for $\Jac C_1$ if and only if it holds for $\Jac C_2$. This is argued identically to Theorem \ref{thm:2PCforEC}.\end{remark}

\begin{remark} Applying remark \ref{re:2PCforhyp} repeatedly proves the 2-parity conjecture for hyperelliptic curves over a number field $K$ with defining polynomial $f_1f_2\cdots f_n$, where each $f_i\in K[x]$ is linear (subject to the usual conditions on the Sturm sequence for each product $f_i\cdots f_{n}$ for $i = 1,\ldots, n-3$). \end{remark}

\end{document}